\documentclass[twoside,11pt]{article}

\usepackage{amssymb,amsmath}
\usepackage{amsthm}
\usepackage{hyperref}
\usepackage{mathrsfs}
\usepackage{textcomp}
\usepackage{color}
\usepackage{marginnote}

\usepackage{verbatim}

\usepackage{sectsty}
\usepackage{fancyhdr}
\makeatletter

\newdimen\bibspace
\renewenvironment{thebibliography}[1]{%
 \section*{\refname 
       \@mkboth{\MakeUppercase\refname}{\MakeUppercase\refname}}%
     \list{\@biblabel{\@arabic\c@enumiv}}%
          {\settowidth\labelwidth{\@biblabel{#1}}%
           \leftmargin\labelwidth
           \advance\leftmargin\labelsep
           \itemsep\bibspace
           \parsep\z@skip     %
           \@openbib@code
           \usecounter{enumiv}%
           \let\p@enumiv\@empty
           \renewcommand\theenumiv{\@arabic\c@enumiv}}%
     \sloppy\clubpenalty4000\widowpenalty4000%
     \sfcode`\.\@m}
    {\def\@noitemerr
      {\@latex@warning{Empty `thebibliography' environment}}%
     \endlist}

\makeatother

\makeatletter

\newtheorem{thm}{Theorem}[section]
\newtheorem{lem}{Lemma}[section]
\newtheorem{prop}{Proposition}[section]
\newtheorem{defn}{Definition}[section]

\newtheorem{cor}{Corollary}[section]
\newtheorem{rem}{Remark}[section]

\def\Xint#1{\mathchoice
  {\XXint\displaystyle\textstyle{#1}}%
  {\XXint\textstyle\scriptstyle{#1}}%
  {\XXint\scriptstyle\scriptscriptstyle{#1}}%
  {\XXint\scriptscriptstyle\scriptscriptstyle{#1}}%
  \!\int}
\def\XXint#1#2#3{{\setbox0=\hbox{$#1{#2#3}{\int}$}
  \vcenter{\hbox{$#2#3$}}\kern-.5\wd0}}

\def\dashint{\Xint-}

\newcommand{\al}{\alpha}                \newcommand{\lda}{\lambda}
\newcommand{\om}{\Omega}                \newcommand{\pa}{\partial}
\newcommand{\va}{\varepsilon}           \newcommand{\ud}{\mathrm{d}}
\newcommand{\be}{\begin{equation}}      \newcommand{\ee}{\end{equation}}
                 
\newcommand{\Lda}{\Lambda}              \newcommand{\B}{\mathcal{B}}
\newcommand{\R}{\mathbb{R}}              \newcommand{\Sn}{\mathbb{S}^n}

\newcommand{\dlim}{\displaystyle\lim}

\begin{document}

\title{\textbf{Compactness of conformal metrics with constant $Q$-curvature. I}
\bigskip}

\author{\medskip YanYan Li\footnote{Supported in part by NSF grants  DMS-1203961 and DMS-1501004.} \  \  and \ \
Jingang Xiong\footnote{Supported in part by NSFC 11501034, a key project of NSFC 11631002 and NSFC 11571019.}}

\date{}

\fancyhead{}
\fancyhead[CO]{Compactness of conformal metrics with constant $Q$-curvature}
\fancyhead[CE]{Y. Y. Li \& J. Xiong}

\fancyfoot{}

\fancyfoot[CO, CE]{\thepage}

\renewcommand{\headrule}{}

\maketitle

\begin{abstract}
We study compactness for nonnegative solutions of the fourth order constant $Q$-curvature equations on smooth compact  Riemannian manifolds of dimension $\ge 5$. If the  $Q$-curvature equals $-1$, we prove that all solutions are
universally bounded. If the  $Q$-curvature is $1$, assuming that Paneitz operator's kernel is trivial and its Green function is positive, we establish universal energy bounds on manifolds which are either locally conformally flat (LCF) or of dimension $\le 9$. Moreover, assuming in addition that a positive mass type theorem holds for the Paneitz operator, we prove compactness in $C^4$. Positive mass type theorems have been verified recently on LCF manifolds or manifolds of dimension $\le 7$, when the Yamabe invariant is positive. We also prove that, for dimension $\ge 8$, the Weyl tensor has to vanish at possible blow up points of a sequence of blowing up solutions. This implies the compactness result in dimension $\ge 8$ when the Weyl tensor does not vanish anywhere.
To overcome difficulties stemming from fourth order elliptic equations, we develop a blow up analysis procedure via integral equations.
\end{abstract}


\section{Introduction}

Let $(M,g)$ be a compact smooth Riemannian manifold of dimension $n\ge 3$, and let
\begin{align}
\label{Paneitz operator}
P_g&:=  \Delta_g^2 -\mathrm{div}_g(a_n R_g g+b_nRic_g)d+\frac{n-4}{2}Q_g, \\
\label{Q-curvature}
Q_g&:=-\frac{1}{2(n-1)} \Delta_g R_g+\frac{n^3-4n^2+16n-16}{8(n-1)^2(n-2)^2} R_g^2-\frac{2}{(n-2)^2} |Ric_g|^2,
\end{align} be the  Paneitz operator \cite{Pan83} and the $Q$-curvature \cite{Bra85}, respectively.  Here $\Delta_g $ denotes the Laplace-Beltrami operator,  $R_g$ and $Ric_g$ denote the scalar curvature and Ricci curvature respectively,  $
a_n=\frac{(n-2)^2+4}{2(n-1)(n-2)} $ and $b_n=-\frac{4}{n-2}.$

In dimension four, the prescribed $Q$-curvature problem has been studied by Chang-Yang \cite{CY95,CY99},  Djadli-Malchiodi \cite{DM}, Li-Li-Liu \cite{LLL} and the references therein. Bubbling analysis and compactness of solutions have been studied by Druet-Robert \cite{DR}, Malchiodi \cite{Mal}, and Weinstein-Zhang \cite{WZhang} among others.

Since the Paneitz operator satisfies, for $n\neq 4$, the transformation law
\be \label{eq:conformal change1}
P_{ u^{\frac{4}{n-4}}g}(\phi)=u^{-\frac{n+4}{n-4}}P_g(u\phi)\quad \forall~ u, \phi\in C^\infty(M), ~u>0
\ee
and $ P_g (1)=\frac{n-4}{2} Q_g$, the constant $Q$-curvature problem is equivalent to solving
\be \label{eq:4th Yamabe}
P_g u=\lda u^{\frac{n+4}{n-4}}, \quad u>0 \quad \mbox{on }M,
\ee
where $\lda\in \{-1,0,1\}$.
The existence of solutions of \eqref{eq:4th Yamabe} with $\lda=1$ was obtained by Qing-Raske \cite{QR0} if $(M,g)$ is locally conformally flat and of positive Yamabe type and its  Poincar\'e exponent is less than $(n-4)/2$,  by Gursky-Malchiodi \cite{GM} if $(M,g)$ has a conformal metric $\tilde g$ satisfying   $R_{\tilde g} \ge 0$, $Q_{\tilde g} \ge 0$ and $Q_{\tilde g} \not\equiv 0 $ simultaneously,  and by Hang-Yang \cite{HY14b} if $(M,g)$ is of positive Yamabe type and has a conformal metric $\tilde g$ satisfying  $Q_{\tilde g} \ge 0$ and $Q_{\tilde g} \not\equiv 0 $. See Gursky-Hang-Lin \cite{GHL} for further results.  Here we say that a compact smooth Riemannian manifold is of positive Yamabe type if the first eigenvalue of the conformal Laplacian $L_g:=-\Delta_g +\frac{(n-2)}{4
(n-1)}R_g$ on $M$,  denoted by $\lda_1(L_g)$, is positive.
The assumptions in every existence result mentioned above imply that
\be
\label{condition:main2}
\mathrm{Ker} P_g=\{0\} \mbox{ and the Green's function }G_g \mbox{ of $P_g$ is positive}.
\ee

In this paper, we study the compactness of solutions of the constant $Q$-curvature equation for $n\ge 5$.
Non-compact examples on manifolds of dimension $n\ge 25$ which are not conformally diffeomorphic  to the standard sphere has been constructed by Wei-Zhao \cite{WZ}.  Multiplicity results have been proved by Bettiol-Piccione-Sire \cite{BPS}.   Compactness of solutions of  the Yamabe equation was studied by Schoen \cite{Schoen91}, Li-Zhu \cite{Li-Zhu99}, Druet \cite{Druet04}, Li-Zhang \cite{Li-Zhang04,Li-Zhang05,Li-Zhang06}, Marques \cite{Marques} and Khuri-Marques-Schoen \cite{KMS} for the positive side if $ n\le 24$ or $M$ is locally conformal flat, and by Brendle \cite{Brendle}, Brendle-Marques \cite{BM} for counterexamples in dimensions $ n \ge 25$. See also Kim-Musso-Wei \cite{KMW} for the counterexamples in the fractional Yamabe setting.

\begin{thm} \label{thm:main theorem} Let $(M,g)$ be a  smooth $n$-dimensional compact Riemannian manifold which is not conformally diffeomorphic   to the standard sphere.
 Suppose that (\ref{condition:main2}) holds and at least one of the following three cases occurs.
\begin{itemize}
\item[(i)] $\lda_1(L_{g})>0$ and $(M,g)$ is locally conformally flat or $n=5,6,7$,
\item[(ii)]  $5\le n\le 9$  and the positive mass type theorem holds for the Paneitz operator,
\item[(iii)] $n\ge 8$ and the Weyl tensor of $g$ does not vanish anywhere on $M$.
\end{itemize}
 For $1< p\le \frac{n+4}{n-4}$, let $u\in C^4(M)$ be a positive solution of
\be \label{eq:main1}
P_g u=c(n)u^{p} \quad \mbox{on }M,
\ee where $c(n)=n(n+2)(n-2)(n-4)$.
Then there exists a constant $C>0$, depending only on $M,g$, and
a positive  lower bound of $p-1$, such that
\be \label{eq:C4 estimate}
\|u\|_{C^4(M)} +\|1/u\|_{C^4(M)}\le C.
\ee
\end{thm}


The positive mass type theorem for Paneitz operator in dimension $8,9$ is understood as in Remark \ref{rem:positive mass}. If $(M,g)$ is locally conformally flat and $p=\frac{n+4}{n-4}$, \eqref{eq:C4 estimate} was established
by Qing-Raske \cite{QR0} under the additional assumptions that $\lda_1(L_g)>0$ and the Poincar\'e exponent
is less than $(n-4)/2$, and also by Hebey-Robert \cite{HR, HR11} with $C$ depending on the $H^2$ norm of $u$,
under the additional assumption  that $P_g$ is coercive. In dimension $n=5$, under the assumption that $R_g>0$ and $Q_g\ge 0$ but not identically equal to zero, \eqref{eq:C4 estimate} was proved independently by Li \cite{Lg}.

Our next theorem gives an energy bound of solutions without assuming either $\lda_1(L_g)>0$ or the positive mass type theorem for Paneitz operator.

\begin{thm}\label{thm:energy} Let $(M,g)$ be a smooth compact Riemannian manifold of dimension $n\ge 5$. Suppose that  \eqref{condition:main2} holds.  Let $u\in C^4(M)$ be a positive solution of \eqref{eq:main1}.
If either $n\le 9$ or $(M,g)$ is
locally conformally flat, then, for some positive constant  $C$ depending  only on $M,g$ and
a positive lower bound of $p-1$,
\[
\|u\|_{H^2(M)}\le C.
\]
\end{thm}

For the negative constant $Q$-curvature equation (and more general ones), we have

\begin{thm}\label{thm:compact1} Let $(M,g)$ be a smooth compact Riemannian manifold of dimension $n\ge 5$. Then for any  $1<p<\infty$, there exists a positive constant $C$, depending only on $M,g$ and $p$, such that every  nonnegative $C^4$ solution of
\be \label{eq:-Q}
P_g(u)=-u^{p} \quad \mbox{on }M
\ee
satisfies
\[
\|u\|_{C^4(M)} \le C.
\]
\end{thm}

The proofs of Theorem \ref{thm:main theorem} and Theorem \ref{thm:energy} make use of important ideas for the proof of compactness of positive solutions  of the Yamabe equation, which were outlined first by Schoen \cite{Schoen89, Schoen89b, Schoen91}, as well as methods developed through the work Li \cite{Li95}, Li-Zhu \cite{Li-Zhu99}, Li-Zhang \cite{Li-Zhang04,Li-Zhang05,Li-Zhang06},
and Marques \cite{Marques}.
Our main difficulty now stems from fourth order equations, for which the sole positivity of solutions is not enough to adapt the above analysis for the second order critical nonlinear elliptic  equations.  For instance, if equation \eqref{eq:main1} is defined only in an open  subset $\om\subset M$ instead of $M$, we do not know yet how to prove the universal $H^2$ estimates in Theorem \ref{thm:energy} on $\om' \subset\subset \om$.  This is very different from the Yamabe equation case.
 We establish a comparison principle  and a blow up analysis procedure for nonlinear integral equations on Riemannian manifolds. This is inspired by our earlier joint work with Jin \cite{JLX, JLX2, JLX3} on the fractional Nirenberg problem.

An earlier version of the present paper, which contains some further results,  can be found in \cite{LX}.

The paper is organized as follows. In section \ref{section:pre}, we set up notations and recall basic results of conformal normal coordinates and Phozaev identities. Theorem \ref{thm:compact1} is proved in section \ref{section:negative}.  In section \ref{s:blowup}, we extend the blow up analysis in \cite{JLX3} to the current setting with lower order terms. Section \ref{section:bubble-expansion} is devoted to a refined blow up analysis. In section \ref{section:Q equation blow}, we prove a sign restriction on Pohozaev integral and vanishing rate of Weyl tensor at isolated simple blow up points.The first two main theorems are proved in section \ref{section:thm1.1}.

\bigskip

\noindent\textbf{Acknowledgments:}
This work was completed while Xiong was a Simons postdoc in Beijing International Center for Mathematical Research. He is deeply indebted to Professor Gang Tian for providing  him the opportunity of working in the excellent and stimulating research environment in BICMR.
 He is also grateful to Professor Jiguang Bao for his support.

\medskip

\section{Preliminaries}

\label{section:pre}

Letters $x,y,z$ denote points in $\R^n$, and capital letters $X,Y,Z$ denote points on Riemannian manifolds.  Denote by $B_r(x)\subset \R^n$ the ball centered at $x$ with radius $r>0$. We often write $B_r$ in replace of $B_r(0)$ for brevity. For $X\in M$, $\B_{\delta}(X)$ denotes the geodesic ball centered at $X$ with radius $\delta$. Throughout the paper, constants $C>0$ in inequalities  may vary from line to line and are universal, which means they depend on given quantities but not on solutions.  $f=O^{(k)}(r^m)$ denotes quantities satisfying $|\nabla^j f(r)|\le C r^{m-j}$ for all  integers $1\le j\le k$, where $k$ is a positive integer and $m$ is a real number. $|\mathbb{S}^{n-1}|$ denotes the area of the standard $(n-1)$-sphere. Here are three constants frequently  used in the paper: $c(n)=n(n+2)(n-2)(n-4)$,
 $\al_n=\frac{1}{2(n-2)(n-4)|\mathbb{S}^{n-1}|}$,
$c_n=c(n)\cdot \al_n =\frac{n(n+2)}{2|\mathbb{S}^{n-1}|}$.

\subsection{Paneitz operator in conformal normal coordinates}

Let $(M,g)$ be a smooth Riemannian manifold of dimension $n\ge 5$, and $P_g$ be the Paneitz operator on $M$.
For any point $\bar X\in M$, it was proved in \cite{LP}, together with some improvement in \cite{Cao} and \cite{Guther}, that there exists a  positive smooth function $\kappa $ (with control) on $M$ such that  the conformal metric $\tilde g=\kappa^{\frac{-4}{n-4}}g$ satisfies, in $\tilde g$-normal coordinates $\{x_1,\dots,x_n\}$ centered at $\bar X$,
\[
\det \tilde g=1 \quad \mbox{in }B_{\delta}
\]  for some $\delta>0$, where
$\kappa(x)=1+O(|x|^2)$ and in particular
\[
\kappa(0)=1, \quad \nabla \kappa(0)=0.
\]
 We refer such coordinates as conformal normal coordinates.   Notice that $\det \tilde g=1+O(|x|^N)$ will be enough for our use if  $N$ is sufficiently large.

    In the $\tilde g$-normal coordinates,
 \begin{align*}
 &R_{ij}(0)=0, \quad  Sym_{ijk} R_{ij,k}(0)=:\frac{1}{6} \sum_{\sigma} R_{\sigma(i),\sigma(j),\sigma(k)} (0)=0, \\&
 R_{,i}(0)=0,\quad \Delta_{ \tilde g} R(0)=-\frac{1}{3!}|W_{\tilde g}(0)|^2,
\end{align*}
 where the Ricci tensor $R_{ij}$, scalar curvature $R$, Weyl tensor $W$ are with respect to $\tilde g$, and $\sigma$ ranges over all the permutations of the set $\{i,j,k\}$.
In addition,  \[
\Delta_{\tilde g} =\Delta+\pa_l \tilde g^{kl}\pa_k+(\tilde g^{kl}-\delta^{kl})\pa_{kl},
\]
\begin{align*}
E(u):&=P_{\tilde g}u-\Delta^2 u
=\frac{n-4}{2}Q_{\tilde g}u+ f^{(1)}_{k}\pa_ku +f^{(2)}_{kl}\pa_{kl}u+f^{(3)}_{kls}\pa_{kls}u+f^{(4)}_{klst}\pa_{klst}u,
\end{align*}
where
\[
| f^{(1)}_{k}(x)|+ | f^{(2)}_{kl}(x)|+ ||x|^{-1}  f^{(3)}_{kls}(x)| +| |x|^{-2} f^{(4)}_{klst}(x)| \le C\sum_{k\ge 1, 2\le k+1\le 4} \|\nabla^k g\|_{L^\infty(B_\delta)}^l.
\]

\begin{lem}\label{lem:GM2.8} In the $\tilde g$-normal coordinates, we have, for any smooth radial function $u$,
\begin{align*}
P_{\tilde g}u=&\Delta^2 u+\frac{1}{2(n-1)}R_{,kl}(0)x^kx^l (c_1^*\frac{u'}{r} +c_2^*u'')-\frac{4}{9(n-2)r^2}\sum_{kl}(W_{ikjl}(0)x^i x^j)^2(u''-\frac{u'}{r})\\&+\frac{n-4}{24(n-1)}|W_g(0)|^2 u+(\frac{\psi_5(x)}{r^2}+\psi_3(x))u'' -(\frac{\psi_5(x)}{r^3} +\frac{\psi_3'(x)}{r}) u'+\psi_1(x)u \\&+O(r^4)u''+O(r^3) u'+O(r^2) u,
\end{align*}
where $r=|x|$, $\psi_k(x),\psi_k'(x)$ are  homogeneous polynomials of degree $k$, and $
c_1^*=\frac{2(n-1)}{(n-2)}-\frac{(n-1)(n-2)}{2}+6-n,~ c_2^*=-\frac{n-2}{2}-\frac{2}{n-2}.$
\end{lem}
\begin{proof} Since $\det \tilde g=1$ and $u$ is radial, we have  $\Delta_{\tilde g}^2 u =\Delta^2 u$. The rest of the proof is the same as that of Lemma 2.8 of \cite{GM}.  It suffices to expand the coefficients of lower order terms  of $P_{\tilde g}$ in Taylor series to a higher order so that $(\frac{\psi_5(x)}{r^2}+\psi_3(x))u'' -(\frac{\psi_5(x)}{r^3} +\frac{\psi_3'(x)}{r}) u'+\psi_1(x)u$ appears.

\end{proof}

\begin{prop}[\cite{GM}, \cite{HY14b}]\label{prop:GM}  Let $\tilde g$ and $B_\delta$ be as before.  Let $G\in C^4(B_\delta\setminus \{0\})$ satisfy $P_{\tilde g} G=\delta_0$, where $\delta_0$ is the Dirac measure at $0$. Then
\begin{itemize}
\item If $n=5,6,7$, or $M$ is flat in a neighborhood of $\bar X$, $
G(x)=\frac{\al_n}{|x|^{n-4}}+A+O^{(4)}(|x|),$
\item If $n=8$, $
G(x)=\frac{\al_n}{|x|^{n-4}}-\frac{\al_n}{1440}|W(\bar X)|^2 \log |x|+O^{(4)}(1),$
\item If $n\ge 9$, $
G(x)=\frac{\al_n}{|x|^{n-4}}\Big(1+\psi_4(x)\Big)+O^{(4)}(|x|^{9-n}),$
\end{itemize}
where $\al_n=\frac{1}{2(n-2)(n-4)|\mathbb{S}^{n-1}|}$, $A$ is a constant, $W(\bar X)$ is the Weyl tensor at $\bar X$, and $\psi_4(x)$ a homogeneous polynomial of degree $4$.
\end{prop}

\begin{cor}\label{cor:Q-gf-expansion} If $P_{\tilde g}G(Z,\cdot)=\delta_{Z}$ in $B_\delta$ for every $Z\in B_{\delta}$.  Then
\[
G(\exp_{\bar X}x, \exp_{\bar X}y)=\frac{\al_n(1+O^{(4)}(|x|^2)+O^{(4)}(|y|^2))}{|x-y|^{n-4}}+\bar a+O^{(4)}(|x-y|^{6-n}),
\]
where $x,y\in B_\delta$, $x-y=(x_1-y_1, \dots, x_n-y_n)$, $|x-y|=\sqrt{\sum_{i=1}^n(x_i-y_i)^2}$, $\bar a$ is a constant and $\bar a=0$ if $n\ge 6$.
\end{cor}

\begin{proof} We only prove the result when $(M,g)$ is not locally formally flat. Denote $X=\exp_{\bar X}x$ and $Y= \exp_{\bar X}y$ for $x,y\in B_\delta$, where $\delta>0$ depends only on $(M,\tilde g)$. For $X\neq \bar X$, we can find $g_{X}=v^{\frac{4}{n-4}} \tilde g$ such that in the $g_{X}$-normal coordinate centered at $X$ there hold $\det g_{X}=1$ and $v(Y)=1+O^{(4)}(dist_{g_{X}}(X,Y)^2)$. Let $G_{g_{X}}$ be the Green's function of $P_{g_{X}}$.   By Proposition \ref{prop:GM},
\[
G_{g_{X}}(X,Y)=\al_{n} dist_{g_{X}}(X,Y)^{4-n}+A+O^{(4)}(dist_{g_{X}}(X,Y)^{6-n}),
\] where  $A$ is a constant and $A=0$ if $n\ge 6$. By the conformal invariance of the Paneitz operator, we have  the transformation law
\[
G(X,Y) =G_{g_{X} }(X,Y) v(X)v(Y)=G_{g_{X} }(X,Y) v(Y).
\]
Since $g_{X}=v^{\frac{4}{n-4}} \tilde g$ and $v(Y)=1+O^{(4)}(dist_{g_{X}}(X,Y)^2)$, we obtain
\[
\begin{split}
dist_{g_{X}}(X,Y)&
=(1+O^{(4)}(|x-y|^2))dist_{\tilde g}(\exp_{\bar X}x,\exp_{\bar X}y)\\&
=(1+O^{(4)}(|x-y|^2))(1+O^{(4)}(|x|^2)+O^{(4)}(|y|^2))|x-y|,
\end{split}
\]
where $\tilde g$ is viewed as a Riemannian metric on $B_\delta$ because of the exponential map $\exp_{\bar X}$.

Therefore, we obtain
\[
G(\exp_{\bar X}x, \exp_{\bar X}y)=\al_{n}\frac{1+O^{(4)}(|\bar x|^2)+O^{(4)}(|y|^2)}{|\bar x-y|^{n-4}}+O^{(4)}(|\bar x-y|^{6-n}).
\]
If $X=\bar X$, the desired estimate follows from Proposition \ref{prop:GM}.
\end{proof}

The following positive mass type theorem for Paneitz operator has been proved through \cite{HuR}, \cite{GM} and \cite{HY14b}.

\begin{thm}\label{thm:positive mass} Let $(M,g)$ be a compact manifold of dimension $n\ge 5$, and $\bar X\in M$. Let  $\tilde g$ be a conformal metric of $g$ such that  $\det \tilde g=1$ in the $\tilde g$-normal coordinate $\{x_1,\dots, x_n\}$ centered at $\bar X$. Suppose also that $\lda_1(L_{g})>0$ and \eqref{condition:main2} holds.  If $n=5,6,7$, or $(M,g)$ is locally conformally flat, then the constant
 $A$ in Proposition \ref{prop:GM} is nonnegative, and $A=0$ if and only if $(M, g)$ is conformally diffeomorphic  to the standard $n$-sphere.
\end{thm}

\begin{rem}\label{rem:positive mass} Suppose  that the assumptions in Theorem  \ref{thm:positive mass} hold. If $W(\bar X)=0$, it follows from Proposition 2.1 of \cite{HY14b} that, in the $\tilde g$-normal coordinates centered at $\bar X$, the Green's function $G$ of $P_{\tilde g}$ has the expansion
\[
G(\bar X, \exp_{\bar X}x)= \begin{cases} \al_8|x|^{-4}+\psi(\theta) +\log |x| O^{(4)}(|x|),& \quad n=8,\\[2mm]
\al_9|x|^{-5}(1+\frac{R_{,ij}(\bar X) x^ix^j |x|^2}{384}) +A+ O^{(4)}(|x|), &\quad n=9,
\end{cases}
\]
where $x=|x|\theta$, $\psi$ is a smooth function of $\theta$, and $A$ is constant. In dimension $n=8,9$, we say the positive mass type theorem holds for Paneitz operator if $\int_{\mathbb{S}^{n-1}} \psi(\theta)\,\ud \theta>0 $ and $A>0$ respectively.
\end{rem}

Let
\[
U_\lda(x):=\left(\frac{\lda}{1+\lda^2|x|^2}\right)^{\frac{n-4}{2}}, \quad \lda>0.
\]
These functions are the unique positive solutions of $\Delta^2 u=c(n)u^{\frac{n+4}{n-4}}$ in $\R^n$, $n\ge 5$, up to translations by Lin \cite{Lin}.  In the $\tilde g$-normal coordinates,   we have, by Lemma \ref{lem:GM2.8}, that
\be \label{eq:cor-GM}
P_{\tilde g} U_{\lda}= c(n)U_{\lda}^{\frac{n+4}{n-4}}+  f_{\lda}U_{\lda},
\ee
where $f_{\lda}(x)$ is a smooth function satisfying that  $\lda^{-k}|\nabla_x^kf_\lda(x)| $, $k=0,1,\dots, 5$, is uniformly bounded in $B_\delta$ by a constant independent of $\lda \ge 1$.
Indeed,
by direct computations,
\begin{align*}
\pa_rU_\lda&=(4-n)\lda^{\frac{n}{2}}(1+\lda^2r^2)^{\frac{2-n}{2}}r,\\
\pa_{rr}^2 U_\lda &=(4-n)(2-n)\lda^{\frac{n+4}{2}}(1+\lda^2r^2)^{\frac{-n}{2}}r^2+(4-n)\lda^{\frac{n}{2}}(1+\lda^2r^2)^{\frac{2-n}{2}}.
\end{align*}
Inserting these to the expression in Lemma \ref{lem:GM2.8} leads to  \eqref{eq:cor-GM}.

\begin{cor}\label{cor:GM2.8} Let $(M,\tilde g)$ be a smooth compact Riemannian manifold of dimension $n\ge 5$, on which $\mathrm{Ker} P_{\tilde g}=\{0\}$. Then there exists a small constant $\delta>0$,  depending only on $(M,\tilde g)$, such that if $\det \tilde g=1$ in the normal coordinates $\{x_1,\dots, x_n\}$ centered at $\bar X$, then
\[
U_{\lda}(x)=c(n)\int_{B_{\delta}}G(\exp_{\bar X} x, \exp_{\bar X} y) \{U_{\lda}(y)^{\frac{n+4}{n-4}}+c_\lda'(y) U_{\lda}(y)\}\,\ud y+c_\lda''(x),
\]
where $\delta>0$ depends only on $M,\tilde g$, and $c_\lda', c_\lda''$ are smooth functions  satisfying
\[
\lda^{-k}|\nabla^kc_\lda'(x)| \le C, \quad |\nabla ^kc_\lda''(x)|\le C\lda^{\frac{4-n}{2}},
\]
for $k=0,1,\dots, 5$ and  some $C>0$ independent of $\lda\ge 1$.
\end{cor}

\begin{proof}  Let $\eta(x)=\eta(|x|)$ be a smooth cutoff function satisfying $
\eta(t)=1 ~ \mbox{for }t<\delta/2, ~\eta(t)=0 ~ \mbox{for }t>\delta.$ By the Green's representation formula, we have
\[
(U_\lda\eta)(x)=\int_{B_\delta} G(\exp_{\bar X} x, \exp_{\bar X} y) P_{\tilde g}(U_\lda\eta)(y)\,\ud y.
\]
Making use of \eqref{eq:cor-GM} and Lemma \ref{lem:GM2.8}, we see that $c_\lda'=\frac{f_\lda}{c(n)}$ and the  proof is finished.
\end{proof}

\subsection{Two Pohozaev type identities}

For $r>0$, define in Euclidean space
\[\begin{split}
\mathcal{P}(r,u):=\int_{\pa B_r}&\frac{n-4}{2}\Big(\Delta u \frac{\pa u}{\pa \nu}-u\frac{\pa }{\pa \nu}(\Delta u)\Big) -\frac{r}{2}|\Delta u|^2
\\&-x^k \pa_k u \frac{\pa }{\pa \nu}(\Delta u)+\Delta u   \frac{\pa }{\pa \nu}(x^k\pa_ku)\,\ud S,
\end{split}
\]
where $\nu=\frac{x}{r}$ is the unit  outward normal to $\pa B_r$.

\begin{prop}\label{prop:4-pohozaev}  Let $u\in C^4(\bar B_r)$ be a positive function and let
\[
E(u)=Ku^{p}-\Delta^2 u \quad \mbox{in }B_r,
\]
where $K\in C^1(\bar B_{r})$ and $p>1$. Then
\begin{align} \label{eq:poho-2.2}
\mathcal{P}(r,u)=& \int_{B_r} (x^k\pa_k u +\frac{n-4}{2} u) E(u)\,\ud x+\mathcal{N}(r,u),
\end{align}
where
\[
\begin{split}
\mathcal{N}(r,u):=& (\frac{n}{p+1}-\frac{n-4}{2}) \int_{B_r} Ku^{p+1}\,\ud x +\frac{1}{p+1}\int_{B_r} x^k\pa_k K  u^{p+1}\,\ud x  \\& -\frac{r}{p+1}\int_{\pa B_{r}} K u^{p+1}\,\ud S.
\end{split}
\]
\end{prop}

\begin{proof} Multiplying both sides of the equation by $u$ and $x\nabla u$ and integrating parts lead to the proposition. We omit the by now standard details. See, e.g.,  Li \cite{Li95},  Marques \cite{Marques}, Hebey-Robert \cite{HR11} for the proofs of the related Pohozaev type identities.

\end{proof}

\begin{lem}\label{lem:test-poho} If $G(x)=|x|^{4-n}+A+O^{(4)}(|x|)$, where $A$ is constant,  then
\[
\lim_{r\to 0}\mathcal{P}(r,G)=-(n-4)^2(n-2)A|\mathbb{S}^{n-1}|.
\]
\end{lem}

The following proposition is a special case of Proposition 2.15 in \cite{JLX3}.
\begin{prop} \label{prop:pohozaev} For $R>0$, let $0\le u\in C^1(\bar B_R)$ be a solution of
\[
u(x)= \int_{B_R} \frac{K(y)u(y)^{p}}{|x-y|^{n-4}}\,\ud y+ h(x),
\]
where $p>0$, $h(x)\in C^1(B_R)$, $\nabla h\in L^1(B_R)$.
Then
\begin{align*}
&\left(\frac{n-4}{2}-\frac{n}{p+1}\right) \int_{B_R} K(x)u(x)^{p+1}\,\ud x-\frac{1}{p+1} \int_{B_R} x\nabla K(x) u(x)^{p+1}\,\ud x \\ &
=\frac{n-4}{2} \int_{B_R} K(x) u(x)^p h(x)\,\ud x+ \int_{B_R} x\nabla h(x) K(x)u(x)^p \,\ud x - \frac{R}{p+1} \int_{\pa B_R} K(x) u(x)^{p+1}\,\ud S.
\end{align*}
\end{prop}

\section{Proof of Theorem \ref{thm:compact1}}

\label{section:negative}

The proof of Theorem \ref{thm:compact1} is divided into two steps.

\emph{Step 1}. $L^p$ estimate. Let $u\ge 0$ be a solution of \eqref{eq:-Q}. Integrating \eqref{eq:-Q} leads to
\[
\int_{M}u^p \ud vol_g =\left| -\int_{M} uP_g(1) \ud vol_g\right| \le \frac{n-4}{2} \|u\|_{L^p(M)} \|Q_g\|_{L^{p'}(M)},
\]
where $\frac{1}{p'}+\frac1p=1$.  It follows that  $\|u\|_{L^p(M)} ^{p-1} \le \frac{n-4}{2} \|Q_g\|_{L^{p'}(M)}$.

\emph{Step 2.} If $\mathrm{Ker} P_{g}=\{0\}$, there exists a unique Green function of $P_g$. If the kernel of $P_{g}$ is non-trivial, since the spectrum of Paneitz operator is discrete, there exists a small constant $\va>0$ such that the kernel of $P_g-\va$ is trivial. Let $G_g$ be the Green function of the operator $P_g-\va$, where $\va\ge 0$. Then there exists a constant $\delta>0$, depending only $M,g$ and $\va$, such that, for every $X\in M$, we have $G(X,Y)>0$ for $Y\in \mathcal{B}_{\delta}(X)$ and $|G_g(X,Y)|\le C(\delta,\va)$ for $Y\in M\setminus \mathcal{B}_{\delta}(X)$.  Rewrite the equation of $u$ as $P_g u-\va u=-(u^p+\va u). $
It follows from the Green representation theorem that
\begin{align*}
u(X)& =-\int_{M}G_g(X,Y)(u^p+\va u)(Y)\ud vol_g(Y)\\&
\le -\int_{M\setminus \mathcal{B}_{\delta}(X)}G_g(X,Y)(u^p+\va u)(Y)\ud vol_g(Y) \\&
\le C\max\{\|u\|_{L^p(M)}^p, \|u\|_{L^p(M)}\} \le C.
\end{align*}
By the arbitrary choice of $X$, we have $\|u\|_{L^\infty}\le C$. The higher order derivatives  estimates follow from standard linear elliptic partial differential equation theories; see Agmon-Douglis-Nirenberg \cite{ADN}.

Theorem \ref{thm:compact1} is proved.  \qed

\section{Blow up analysis for integral equations}
\label{s:blowup}

As mentioned in the Introduction, we will prove Theorem \ref{thm:main theorem} and Theorem \ref{thm:energy} by studying nonlinear integral equations on Riemannian manifolds. This section is devoted to the local theory of blow up solutions of nonlinear integral equations. This is inspired by Jin-Li-Xiong \cite{JLX3}, where the manifold is the unit sphere and the stereographic projection was used to have integral equations involving the standard Riesz kernels.

 We will use $A_1, A_2,A_3 $ to denote positive constants,  and $\{\tau_i\}_{i=1}^\infty$ to denote  a sequence of nonnegative constants satisfying $\lim_{i\to \infty}\tau_i=0$.  Set
\be \label{p}
p_i=\frac{n+4}{n-4}-\tau_i.
\ee
Let $\{G_i(x,y)\}_{i=1}^\infty$ be a sequence of functions on $B_3\times B_3$ satisfying
\be  \label{G}
\begin{aligned}
&G_i(x,y)=G_i(y,x),\qquad  G_i(x,y)\ge A_1^{-1}|x-y|^{4-n},\\[2mm]
& |\nabla^l_x G_i(x,y)|\le A_1|x-y|^{4-n-l}, \quad l=0,1,\dots, 5 \\[2mm]
 &G_i(x,y)=c_{n}\frac{1+O^{(4)}(|x|^2)+O^{(4)}(|y|^2)}{|x-y|^{n-4}}+\bar a_i+O^{(4)}(\frac{1}{|x-y|^{n-6}})
\end{aligned}
\ee
for all $x,y\in \bar B_3$, where $
c_n=\frac{n(n+2)}{2|\mathbb{S}^{n-1}|}
$ is the constant given towards the end of the introduction, $f=O^{(4)}(r^m)$ denotes some quantity satisfying $|\nabla^j f(r)|\le A_1 r^{m-j}$ for all  integers $1\le j\le 4$,  and $\bar a_i$ is a constant and $\bar a_i=0$ if $n\ge 6$.
Let  $\{K_i\}_{i=1}^\infty \in C^\infty(\bar B_3)$ satisfy
\be \label{K}
\lim_{i\to \infty}K_i(0)=1, \quad K_i\ge A_2^{-1} , \quad
\| K_i\|_{C^5(B_3)} \le A_2.
\ee
Let  $\{ h_i\}_{i=1}^\infty$ be a sequence of nonnegative functions in $C^\infty( B_3)$ satisfying for $0<r<1/2$,
\be \label{H}
\begin{aligned}
\max_{\bar B_{r}(x)} h_i \le A_2 \min_{\bar B_r(x)} h_i \quad \mbox{and }
\sum_{j=1}^5r^j|\nabla^j h_i(x)| \le A_2 \|h_i\|_{L^\infty(B_r(x))}, \quad x\in B_{2}.
 \end{aligned}
\ee

Given $p_i, G_i, K_i, $ and $ h_i$ satisfying \eqref{p}-\eqref{H}, let $0\le u_i\in L^{\frac{2n}{n-4}}(B_3)$
be a solution of
\be \label{eq:s1}
u_i(x)=\int_{B_3 }G_i(x,y) K_i(y) u_i^{p_i}(y)\,\ud y +h_i(x) \quad \mbox{in  } B_3.
\ee
It follows from \cite{Li04} and Proposition \ref{prop:local estimates} that $u_i\in C^{4}(B_3)$. In the following we will always assume $u_i\in C^{4}(B_3)$. We say that $\{u_i\}$ blows up if $\|u_i\|_{L^\infty(B_3)}\to \infty$ as $i\to \infty$.

\begin{defn}\label{def4.1}
We say that $x_i\to \bar x\in B_3$ is an isolated blow up point of $\{u_i\}$ if there exist
$0<\overline r<dist(\overline x,\pa B_3)$, $\overline C>0$, such that,
$x_i$ is a local maximum of $u_i$, $u_i(x_i)\to \infty$ and
\be\label{4.7}
u_i(x)\leq \overline C |x-x_i|^{-4/(p_i-1)} \quad \mbox{for all } x\in B_{\overline r}(x_i).
\ee
\end{defn}

Let $x_i\to \overline x$ be an isolated blow up of $u_i$. Define
\[
\overline u_i(r)=\frac{1}{|\pa B_r|} \int_{\pa B_r(x_i)}u_i\,\ud S \quad \mbox{and } \quad  \overline w_i(r)=r^{4/(p_i-1)}\overline u_i(r), \quad  r>0,
\]

\begin{defn}\label{def4.2}
We say $x_i \to \overline x\in B_3$ is an isolated simple blow up point, if $x_i \to \overline x$ is an isolated blow up point, such that, for some
$\rho>0$ (independent of $i$) $\overline w_i$ has precisely one critical point in $(0,\rho)$ for large $i$.
\end{defn}

The notations of isolated blow up points and isolated  simple blow up points were introduced by Schoen \cite{Schoen89, Schoen89b, Schoen91}.
If $x_i=\bar x$ in the above definitions, we will say $\bar x$ is an isolated or isolated simply blow up point for short. The conclusion of the following lemma was proved by  Djadli-Malchiodi-Ould Ahmedou \cite{DMO} for solutions of the equation $\Delta^2 u_i= K_i u_i^{p_i}$ under the assumption that $\Delta u_i<0$, where $K_i $ is uniformly bounded from below and above by positive constants and $1<p_i\le \frac{n+4}{n-4}$.

\begin{lem}\label{lem:harnack}  Given $p_i, G_i, K_i$ and $h_i$ satisfying \eqref{p}-\eqref{H}, let $0\le u_i\in C^{4}(B_3)$ be a solution of \eqref{eq:s1}.  Suppose that $0$ is an isolated blow up point of $\{u_i\}$  with $\bar r=2$ and  $\bar C= A_3$.
Then for any $0<r<1/3 $ we have
\[
\sup_{B_{2r}\setminus B_{r/2}} u_i\leq C \inf_{B_{2r}\setminus B_{r/2}}  u_i,
\]
where $C$ is a positive constant depending only on $n, A_1, A_2, A_3$.
\end{lem}

\begin{proof}  For every $0<r<1/3$, set
\[
w_i(x)=r^{4/(p_i-1)}u_i(rx), \quad \tilde h_i(x):=r^{4/(p_i-1)} h_i(rx),
\]
\[
G_{i,r}(x,y)=r^{n-4}G_i(rx,ry) \quad \mbox{for }r>0.
\]
By the equation of $u_i$, we have
\[
w_i(x)= \int_{B_{3/r}} G_{i,r}(x,y)K_i(ry)w_i(y)^{p_i}\,\ud y + \tilde h_i(x)\quad \mbox{for } x\in B_{3/r}.
\]
Since $0$ is an isolated blow up point of $u_i$,
\be \label{eq:out}
w_i(x)\leq A_3 |x|^{-4/(p_i-1)} \quad \mbox{for all  }x\in B_3.
\ee
Set $\om_1=B_{5/2}\setminus  B_{1/4}$, $\om_2=B_{2}\setminus  B_{1/2}$ and $
V_i(y)=K_i(ry)w_i(y)^{p_i-1}$. Thus $w_i$ satisfies the linear equation
\[
w_i(x)= \int_{\om_1} G_{i,r}(x,y)V_i(y)w_i(y)\,\ud y + \bar h_i(x)\quad \mbox{for } x\in B_{5/2}\setminus  B_{1/4},
\]
where
\[
\bar h_i(x)=\tilde h_i(x)+\int_{B_{3/r}\setminus \om_1} G_{i,r}(x,y)K_i(ry)w_i(y)^{p_i}\,\ud y.
\]
By \eqref{eq:out}  and \eqref{K}, $\|V_i\|_{L^\infty(\om_1)} \le C(n,A_1,A_2,A_3)<\infty$.
Since $K_i$ and $w_i$ are nonnegative, by  \eqref{G} on $G_i$ and \eqref{H} on $h_i$ we have $\max_{\bar \om_2}\bar h_i \le C(n,A_1,A_2) \min_{\bar \om_2} \bar h$. Applying Proposition \ref{prop:har} to $w_i$ gives $
\max_{\bar \om_2} w_i \le C \min_{\bar \om_2} w_i,$
where $C>0$ depends only on $n, A_1,A_2 $ and $A_3$. Rescaling back to $u_i$, the lemma follows.
\end{proof}

\begin{prop}\label{prop:blow up a bubble}
Suppose that $0\le u_i\in C^{4}(B_3)$ is a solution of \eqref{eq:s1} and all assumptions in Lemma \ref{lem:harnack} hold.
Let $R_i\rightarrow \infty$ with $R_i^{\tau_i}=1+o(1)$ and $\va_i\rightarrow 0^+$, where $o(1)$ denotes some quantity tending to $0$ as $i\to \infty$. Then we have,
after passing to a subsequence,
with  $m_i:=u_i(0)$,
\[
\|m_i^{-1}u_i(m_i^{-(p_i-1)/4} \cdot)-(1+ |\cdot|^2)^{(4-n)/2}\|_{C^3(B_{2R_i}(0))}\leq \va_i,
\]
\[
r_i:=R_im_i^{-(p_i-1)/4}\rightarrow 0\quad \mbox{as}\quad i\rightarrow \infty.
\]

\end{prop}

\begin{proof} Set $\hat h_i(x) =\int_{B_3\setminus B_2}G_i(x,y) K_i(y) u_i^{p_i}(y)\,\ud y +h_i(x) $. By \eqref{G} and \eqref{H},  $\hat h_i$ also satisfies \eqref{H} for all $x\in B_1$.    Let
\[
\varphi_i(x)=m_i^{-1} u_i(m_i^{-(p_i-1)/4} x) \quad \mbox{for }|x|<2 m_i^{(p_i-1)/4}.
\]
By the equation of $u_i$,  we have,
\be \label{eq:scal}
\varphi_i(x)= \int_{B_{2 m_i^{(p_i-1)/4}}} \tilde G_i (x,y) \tilde K_i(y)\varphi_i(y)^{p_i}\,\ud y+\tilde h_i(x),
\ee
where $\tilde G_i (x,y)=G_{i,m_i^{-\frac{p_i-1}{4}}  }(x,y)$, $\tilde K_i(y)=K_i(m_i^{-\frac{p_i-1}{4}} y )$ and $\tilde h_i(x)= m_i^{-1} \hat h_i(m_i^{-\frac{p_i-1}{4}} x)$.

First of all, $\max_{\pa B_1 } \hat h_i\le \max_{\pa B_1 } u_i\le A_3$,  by  \eqref{H} for $\hat h_i$ we have
\be \label{eq: h goes}
\tilde h_i \to 0\quad \mbox{in } C^5_{loc}(\R^n) \quad \mbox{as } i\to \infty.
\ee

Secondly, since $0$ is an isolated blow up point of $u_i$,
\be \label{eq:scalbound}
\varphi_i(0)=1, \quad \nabla \varphi_i(0)=0, \quad  0<\varphi_i(x)\leq A_3 |x|^{-4/(p_i-1)}.
\ee
For any $R>0$, we claim that for sufficiently large $i$,
\be \label{eq:scalbound2}
\|\varphi_i\|_{C^{4}(B_R)}\leq C(R).
\ee

Indeed, by Proposition \ref{prop:local estimates} and \eqref{eq:scalbound}, it suffices to prove that $\varphi_i \le C$ in $B_1$.
If $\varphi_i(\bar x_i)=\sup_{B_{1}} \varphi_i \to \infty $, set
\[
\tilde \varphi_i(z)=\varphi_i(\bar x_i)^{-1}\varphi_i(\varphi_i(\bar x_i)^{-(p_i-1)/4}z+\bar x_i)\leq 1\quad \mbox{for }|z| \le \frac12 \varphi_i(\bar x_i)^{(p_i-1)/4}.
\]
By \eqref{eq:scalbound},
\[
\tilde \varphi_i(z_i)= \varphi_i(\bar x_i)^{-1} \varphi_i(0)\to 0
\]  for $z_i=-\varphi_i(\bar x_i)^{(p_i-1)/4}\bar x_i$.  Since $\varphi_i(\bar x_i) \leq  A_3 |\bar x_i|^{-4/(p_i-1)}$, we have $|z_i|\leq A_3^{4/(p_1-1)}$. Hence, we can find $t>0$ independent of $i$ such that such that $z_i\in B_t$.
 Applying Proposition \ref{prop:har} to $\tilde \varphi_i$ in $B_{2t}$ (since $\tilde\varphi_i$ satisfies a similar equation to \eqref{eq:scal}), we have $
1=\tilde \varphi_i(0)\le C \tilde \varphi_i(z_i)\to 0,$
which is impossible. Hence, $\varphi_i \le C$ in $B_1$.

It follows from \eqref{eq:scalbound2} that there exists a function $\varphi\in C^4(\R^n)$ such that, after passing subsequence,
\be\label{eq:phigeos}
\varphi_i(x)\to \varphi \quad \mbox{in }C^3_{loc}(\R^n)  \quad \mbox{as }i\to \infty.
\ee
For any fixed $x$ and  $R>2|x|$, by \eqref{G} and \eqref{K} and \eqref{eq:scalbound} we have
\begin{align}
\int_{B_{2 m_i^{(p_i-1)/4}}\setminus B_{R}}  \tilde G_i (x,y) \tilde K_i(y)\varphi_i(y)^{p_i}\,\ud y& \le C\int_{B_{2 m_i^{(p_i-1)/4}}\setminus B_{R}}  |y|^{4-n} |y|^{-\frac{4p_i}{p_i-1}}\,\ud y \nonumber  \\&
=O(R^{-\frac{n-4}{2}+O(\tau_i)}).
\label{eq:ggoes}
\end{align}
By \eqref{G} and \eqref{K},  we have
\[
\tilde G_i(x,y)\to c_n \frac{1}{|x-y|^{n-4}}\quad \forall~ x\neq y, \quad \mbox{and }\tilde K_i(y)\to K_i(0)=1.
\]
 Combining \eqref{eq: h goes}, \eqref{eq:phigeos} and \eqref{eq:ggoes} together, by \eqref{eq:scal} we have that for any fixed $R>0$ and $x\in B_{R/2}$
\be \label{eq:limit1}
\varphi(x)=c_{n}\int_{B_R}\frac{ \varphi(y)^{\frac{n+4}{n-4}}}{|x-y|^{n-4}}\,\ud y+O(R^{-\frac{4-n}{2}}).
\ee

Sending $R\to \infty$,  it follows from Lebesgue's monotone convergence theorem that
\[
\varphi(x)=c_{n} \int_{\R^n}\frac{\varphi(y)^{\frac{n+4}{n-4}}}{|x-y|^{n-4}} \ud y \quad x\in \R^n.
\]
By the classification theorem in \cite{CLO} or \cite{Li04}, using $\varphi(0)=1$ and $\nabla \varphi(0)=0$, we have
\[
\varphi(x)=\left(1+|x|^2\right)^{-\frac{n-4}{2}}.
\]
\end{proof}
Since passing to subsequences does not affect our proofs, in the rest of the paper we will always choose $R_i\to\infty$ with $R_i^{\tau_i}=1+o(1) $ first, and then $\va_i\to 0^+$ as small as we wish (depending on $R_i$) and then choose our subsequence $\{u_{j_i}\}$ to work with. Since $i\le j_i$ and $\lim_{i\to \infty}\tau_i=0$, one can ensure that $R_i^{\tau_{j_i}}=1+o(1) $ as $i\to \infty$. In the sequel, we will still denote the subsequences as $u_i, \tau_i$ and etc.

\begin{rem}\label{rem:blow} By checking the proof of Proposition \ref{prop:blow up a bubble}, together with the fact $\nabla^2 (1+|x|^2)^{-\frac{n-4}{2}}$ is negatively definite near zero and the $C^2$ convergence of the rescaled profiles in a fixed neighborhood of zero, the following statement holds. Let $0\le u_i\in C^{4}(B_3)$ be a solution of \eqref{eq:s1} and satisfy \eqref{4.7}.  Suppose that $u_i(0)\to \infty$ as $i\to \infty$, $\nabla u_i(0)=0$ and $\max_{B_3}u_i\le b u_i(0)$ for some constant $b\ge 1$ independent of $i$. Then, after passing to a subsequence, $0$ must be a local maximum point of $u_i$ for $i$ large. Namely, $0$ is an isolated blow up point of $u_i$ after passing to a subsequence.

\end{rem}

\begin{prop}\label{prop:lower bounded by bubble} Under the hypotheses of Proposition \ref{prop:blow up a bubble}, there exists a constant $C>0$, depending only on $n, A_1,A_2$ and $A_3$, such that,
\[
u_i(x)\geq C^{-1}m_i(1+m_i^{(p_i-1)/2}|x|^2)^{(4-n)/2}, \quad |x|\leq 1.
\]
In particular, for any $e\in \mathbb{R}^n$, $|e|=1$, we have
\[
u_i(e)\geq C^{-1}m_i^{-1+((n-4)/4)\tau_i}.
\]
\end{prop}
\begin{proof} By Proposition \ref{prop:blow up a bubble}, with $U_{1 }(x)= (1+ |x|^2)^{-\frac{n-4}{2}}$,
\begin{align*}
 u_i(x)&\ge \frac{1}{C} \int_{B_{r_i}} |x-y|^{4-n} u_i(y)^{p_i}\,\ud y  \\&
 \ge  \frac{1}{C} m_i \int_{B_{R_i}} |x-y|^{4-n} U_1(y)^{p_i}\,\ud y \ge \frac{1}{2} C m_i U_1(m_i^{(p_i-1)} x).
 \end{align*}
\end{proof}

\begin{lem}[Comparison principle] \label{lem:cp} Let $\om_1 \subset \om$ be two open  sets in $\R^n$ and $\om_1$ is bounded.  Let $G\in L^{1}(\om \times \om)$, $G>0$ a.e in $\om\times \om$, and satisfy that for some small constant $\beta >0$,
\[
\al:=\operatorname{ess}\inf_{x\in \om_1} \inf_{\substack{ E\subset \om_1,\\ |E|\ge |\om_1|-\beta}} \int_{E}G(x,y)\,\ud y>0,
\]
 where $| E|$ denotes the Lebesgue measure of set $E$.  Let $f\in  C(\bar \om_1\times \R) $ and $\frac{f(x, s)}{s}$ be nondecreasing in $s\in (0,\infty)$ for each $x\in   \om_1 $.  Define $L u(x)= \int_{\om} G(x,y)u (y)\,\ud y$ for $u\in L^\infty( \om)$. If $u,v \in  L^\infty (\om)$ satisfy
\be\label{eq:inin}
\max\{ 0, f(x,u(x))- Lu(x)\} \le f(x,v(x))-Lv(x) \quad a.e. \mbox{ in } \om_1,
\ee
\be\label{eq:essinf++}
 u\le v \quad  a.e. \mbox{ in }\om \setminus \om_1, ~ v\ge 0  \quad a.e. \mbox{ in } \om  \quad \mbox{and }\operatorname{ess}\inf_{\om_1} v>0,
\ee
then either  $ u\le v$ a.e. in $\om_1$ or $u=v=0$ a.e. in $\om\setminus \om_1$, $u=tv$ and $f(\cdot, tv)=tf(\cdot, v)$ a.e. in $\om_1$ for some $t>1$.

\end{lem}

\begin{proof} Since $u\in L^\infty(\om_1)$, by \eqref{eq:essinf++} we see that
\[
\bar t=\inf \{t>1: tv\ge u \mbox{ a.e. in }\om_1\}
\]
is well defined and finite. Suppose $\bar t>1$, otherwise we are done.   If $|\{\bar t v-u=0\} \cap \om_1|>0$, by the hypotheses of the lemma  one can find $\bar x\in \om_1$ such that $t v(\bar x)=u(\bar x)>0$, \eqref{eq:inin} holds at $\bar x$ and $G(\bar x,\cdot)>0$ a.e. in $\om$.  It follows that
\begin{align*}
0&=f(\bar x, \bar t v(\bar x))- f(\bar x, u(\bar x))= \frac{f(\bar x, \bar t v(\bar x))}{\bar t v(\bar x)}\bar t v(\bar x)-f(\bar x, u(\bar x)) \\&
\ge \bar tf(\bar x, v(\bar x))- f(\bar x, u(\bar x))
\ge L(\bar t v-u)(\bar x) \ge 0
\end{align*}
and thus $\bar tv-u=0$ a.e. in $\om$. Since $\bar t>1$, by \eqref{eq:essinf++} it must have $u=v=0$ a.e. in $\om\setminus \om_1$. Since the above holds for almost every $\bar x\in \om_1$, we see that $f(\cdot, tv)=tf(\cdot, v)$ a.e. in $\om_1$.

If $|\{\bar t v-u=0\} \cap \om_1|=0$, one can find $\delta>0$ such that  $| \{\bar t v-u>\delta\}\cap \om_1| \ge |\om_1|-\beta $. For every sufficiently small  $\va>0$, by the definition of $\bar t$ we have $|\{\bar t v-u<\va\} \cap \om_1| >0$.  By the hypotheses of the lemma   one can find $\bar x\in \om_1$ such that $0\le t v(\bar x)-u(\bar x)<\va$, \eqref{eq:inin} holds at $\bar x$, $G(\bar x,\cdot)>0$ a.e. in $\om$ and $ \inf_{\substack{ E\subset \om_1,\\ |E|\ge |\om_1|-\beta}} \int_{E}G(\bar x,y)\,\ud y \ge \al$. It follows that
\begin{align*}
o(1)=f(\bar x, \bar t v(\bar x))- f(\bar x, u(\bar x))&
\ge \bar tf(\bar x, v(\bar x))- f(\bar x, u(\bar x))\\&
\ge L(\bar t v-u)(\bar x)  \ge \int_{\{\bar t v-u>\delta\}} G(\bar x, y) \delta \,\ud y\ge \al \delta>0,
\end{align*}
where $o(1)\to 0$ as $\va>0$. Sending $\va \to$, we obtain a contradiction.
The lemma is proved.
\end{proof}

\begin{lem} \label{lem:upbound1}  Suppose that the hypotheses of Proposition \ref{prop:blow up a bubble} hold and in addition that $ 0$ is also an isolated simple blow up point with constant $\rho>0$. Then there exist $\delta_i>0$, $\delta_i=O(R_i^{-4})$,
such that
\[
u_i(x)\leq C u_i(0)^{-\lda_i}|x|^{4-n+\delta_i},\quad \mbox{for all }r_i\leq |x|\leq 1,
\]
where $\lda_i=(n-4-\delta_i)(p_i-1)/4-1$ and $C>0$ depends only on $n, A_1,A_2, A_3$ and $\rho$.

\end{lem}

\begin{proof}  We divide the proof into several steps. Unless otherwise stated,  we will use $C$ to denote some constant depending only on $n, A_1,A_3$ and $\rho$, which may change from line to line.

\textbf{Step 1.} From Proposition \ref{prop:blow up a bubble}, we see that
\begin{align}
u_i(x)&\le C m_i \left(\frac{1}{1+|m_i^{(p_i-1)/4}x|^2}\right)^{\frac{n-4}{2}}
\le C m_iR_i^{4-n} \quad \mbox{for all } |x|=r_i=R_i m_i^{-(p_i-1)/4}.
                             \label{4.8}
\end{align}
Let $\overline u_i(r)$ be the average of $u_i$ over the sphere of radius $r$ centered at $0$.
It follows from the assumption of isolated simple
blow up points and Proposition \ref{prop:blow up a bubble} that
\be\label{4.9}
r^{4/(p_i-1)}\overline u_i(r) \quad \mbox{is strictly decreasing for $r_i<r<\rho$}.
\ee
By Lemma \ref{lem:harnack},  \eqref{4.9} and \eqref{4.8},  we have, for all $r_i<|x|<\rho$,
\[
\begin{split}
|x|^{4/(p_i-1)}u_i(x)&\leq C|x|^{4/(p_i-1)}\overline u_i(|x|)
\leq C r_i^{4/(p_i-1)}\overline u_i(r_i)
\leq CR_i^{\frac{4-n}{2}},
\end{split}
\]
where we used $R_i^{\tau_i}=1+o(1)$.  Thus,
\be \label{eq:coeff}
u_i(x)^{p_i-1}\leq C R_i^{-4}|x|^{-4} \quad \mbox{for all } r_i\leq |x|\le \rho.
\ee

\medskip

\textbf{Step 2.}  Let
\[
\mathcal{L}_i\phi(y):= \int_{B_3} G_i(x,z)K_i(z) u_i(z)^{p_i-1}\phi(z)\,\ud z.
\]
Thus
\be \label{eq:linearize}
u_i=\mathcal{L}_i u_i+h_i.
\ee
Set $M_i:=4^n A_1^2\max_{\pa B_\rho} u_i+2\max_{\bar B_\rho}h_i$, and, with $A>1$ being chosen later,
\[
f_i(x):=M_i \rho^{\delta_i} |x|^{-\delta_i}+A  m_i^{-\lda_i} |x|^{4-n+\delta_i},
\]
\[
\phi_i(x)=\begin{cases}
f_i(x), & \quad r_i\le  |x|\le  \rho,\\
u_i(x),&\quad \mbox{otherwise} .
\end{cases}
\]

Making use of Proposition \ref{prop:blow up a bubble}, we have for $r_i<|x|<\rho $  that
\begin{align} \label{eq:first-part}
\int_{|y|\le r_i} \frac{u_{i}(y)^{p_i}}{|x-y|^{n-4}} \,\ud y \le Cm_i  U_1(m_i^{(p_i-1)/4}x)
\le Cm_i^{-\lda_i}|x|^{4-n+\delta_i}.
\end{align}

Note that for $4<\mu<n$ and $0<|x|<2$,
\begin{align*}
\int_{B_3} G_i(x,y)|y|^{-\mu}\,\ud y&\le A_1\int_{\R^n} \frac{1}{|x-y|^{n-4}|y|^{\mu}}\,\ud y
 \le C\Big( \frac{1}{n-\mu}+\frac{1}{\mu- 4} \Big)|x|^{-\mu+4}.
\end{align*}
By \eqref{eq:coeff}, one can properly choose $0<\delta_i=O(R_i^{-4})$ such that for all $r_i< |x|<\rho$,
\be \label{eq:a1}
\int_{r_i<|y|<\rho} G_i(x,y)K_i(y)u_i(y)^{p_i-1} f_i(y) \,\ud y\leq \frac{1}{4} f_i(x).
\ee
Finally, for every $r_i< |x|<\rho$,  by \eqref{G} we have
\begin{align}
\int_{\rho\le |y|<3} G_i(x,y)K_i(y)u_i(y)^{p_i-1} \phi_i(y)\,\ud y&\le A_1^2 2^{n+4}\int_{\rho\le |y|<3} G_i(\frac{\rho x}{|x|},y)K_i(y)u_i(y)^{p_i}\,\ud y \nonumber \\ &\le  A_1^2 2^{n+4} u_i(\frac{\rho x}{|x|}) \le 2^{4-n} M_i,
\label{eq:third-part}
\end{align}
where we have used \eqref{eq:s1} at $\rho x/|x|$ and the positivity of $h_i$.

By \eqref{eq:first-part}, \eqref{eq:a1} and \eqref{eq:third-part}, we have for $r_i<|x|< \rho$.
\begin{align}
\mathcal{L}_i \phi_i(x)+h_i(x) & \leq Cm_i^{-\lda_i}|x|^{4-n+\delta_i}  +\frac{f_i(x)}{ 4}+2^{4-n} M_i+ \max_{\bar B_\rho}h_i
\le  f_i(x)=\phi_i(x)
\label{eq:intineq}
\end{align}
by choosing $A$ large. Choosing $A$ to  be larger than the constant $C$ in \eqref{eq:coeff}, we have for $|x|=r_i$
\[
f_i(x)  >A m_i^{-\lda_i}R_i^{4-n+\delta_i} m_i^{(p_i-1)(n-4-\delta_i)/4}=A R_i^{4-n+\delta_i} m_i>u_i(x).
\]  Thus, using \eqref{eq:linearize}, \eqref{eq:intineq}, the definition of $f_i$ and $M_i$, and the fact that $\phi_i \ge u_i$ on $\pa (B_\rho \setminus B_{r_i})$, we apply Lemma \ref{lem:cp} to obtain
\be \label{eq:upper bounded}
u_i(x)\le \phi_i(x).
\ee

\textbf{Step 3.}  By \eqref{H}, we have $\max_{\bar B_{\rho}} h_i \le A_2\max_{\pa \bar B_{\rho}}h_i \le A_2 \max_{\pa \bar B_{\rho}}u_i$.
 Hence, $
 M_i\le C\max_{\pa B_{\rho}} u_i . $
For $r_i<\theta<\rho$,  we have, using \eqref{eq:upper bounded}, that
\[
\begin{split}
\rho^{4/(p_i-1)}M_i&\leq C \rho^{4/(p_i-1)}\overline u_i(\rho)
\leq C\theta^{4/(p_i-1)}\overline u_i(\theta)
\leq C\theta^{4/(p_i-1)}\{M_i\rho^{\delta_i}\theta^{-\delta_i}+Am_i^{-\lda_i}\theta^{4-n+\delta_i}\}.
\end{split}
\]
Choose $\theta=\theta(n,\rho,A_1,A_2, A_3)$ sufficiently small so that
\[
C\theta^{4/(p_i-1)}\rho^{\delta_i}\theta^{-\delta_i}\leq \frac12 \rho^{4/(p_i-1)}.
\]
Hence, we have $
M_i\le C m_i^{-\lda_i}.$
It follows from \eqref{eq:upper bounded} that
\[
u_i(x)\le \phi_i(x) \le Cm_i^{-\lda_i}|x|^{-\delta_i}+A m_i^{-\lda_i} |x|^{4-n+\delta_i} \le Cm_i^{-\lda_i} |x|^{4-n+\delta_i}.
\]\end{proof}

\begin{lem} \label{lem:aux1} Under the assumptions in Lemma \ref{lem:upbound1}, for $k<n$ we have
\[
I_k[u_i^{p_i}](x) \le C\begin{cases}
 m_i^{\frac{n-2k+4}{n-4}+o(1)}, &\quad \mbox{if } |x|<r_i, \\
 m_i^{-1+o(1)}|x|^{k-n}, &\quad \mbox{if }r_i \le |x|<1,
 \end{cases}
\]
where
\be \label{eq:localriesz}
I_k[v](x) =\int_{B_1}|x-y|^{k-n} v\,\ud y \quad \mbox{for }v\in L^1(B_1).
\ee
\end{lem}

\begin{proof} Making use of Proposition \ref{prop:blow up a bubble} and Lemma \ref{lem:upbound1},  we have
\begin{align*}
I_k[u_i](x) & = \int_{B_{r_i}}\frac{u_i(y)^{p_i}}{ |x-y|^{n-k}}  \,\ud y+ \int_{B_1\setminus B_{ r_i}} \frac{u_i(y)^{p_i}}{ |x-y|^{n-k}}\,\ud y\\&
\le Cm_i^{\frac{n-2k+4}{n-4}+o(1)} \int_{B_{R_i}} \frac{U_1(z)^{p_i}\,\ud z }{|m_i^{(p_i-1)/4} x-z|^{n-k}} +Cm_i^{-\frac{n+4}{n-4}+o(1)}\int_{B_1\setminus B_{ r_i}}\frac{\,\ud y}{|x-y|^{n-k}|y|^{n+4}}.
\end{align*}
\end{proof}

\begin{lem}\label{lem:error}
Under the assumptions in Lemma \ref{lem:upbound1},  we have
\[
 \tau_i=O(u_{i}(0)^{-2/(n-4)+o(1)}), \quad \mbox{ and consequently} \quad  m_i^{\tau_i}=1+o(1).
\]

\end{lem}

\begin{proof}  For $x\in B_1$, we write equation \eqref{eq:s1} as
\be \label{eq:flat}
u_i(x)= c_{n}\int_{B_1} \frac{K_i(y)u_i(y)^{p_i}}{|x-y|^{n-4}}\,\ud y +b_i(x),
\ee
where $b_i(x):=Q'_i(x)+Q''_i(x) +h_i(x)$,
\begin{align*}
Q_i' (x):&=\int_{B_1}(G_i(x,y)-c_{n}|x-y|^{4-n}) K_i(y) u_i(y)^{p_i}\,\ud y,\\
Q_i''(x):&= \int_{B_3\setminus B_1} G_i(x,y)K_i(y) u_i(y)^{p_i}\,\ud y.
\end{align*}
Notice that
\[
|G_i(x,y)-c_{n}|x-y|^{4-n}| \le \frac{C|x|^2 }{|x-y|^{n-4}}+|\bar a_i|+C|x-y|^{6-n}.
\]
\[
|\nabla_x(
G_i(x,y)-c_{n}|x-y|^{4-n})| \le \frac{C|x|^2 }{|x-y|^{n-3}}+\frac{C|x| }{|x-y|^{n-4}}+C|x-y|^{5-n}.
\]
Hence, with $I_k [\cdot] $ denoting the Riesz potential in \eqref{eq:localriesz},
\begin{align*}
|Q_i'(x)| & \le C(|x|^2 u_i(x)+|a_i| \|u_i^{p_i}\|_{L^1(B_1)}+I_{6} [u_i^{p_i}](x)), \\
|\nabla Q_i'(x)|  &\le C ( |x|^2 I_{3}+|x|I_{4} +I_{5}) [u_i^{p_i}](x).
\end{align*}

By Lemma \ref{lem:upbound1}, we have  $u_i(x)\le  C m_i^{-\lda_i}$ for all $x\in B_{3/2}\setminus B_{1/2}$. Hence, $Q_i''(x) +h_i(x)\le u_i(x)\le C m_i^{-1+o(1)}$ for any $x\in \pa B_1$ . It follows from \eqref{H} that
\[
\max_{\bar B_2}h_i(x) \leq C \min_{\pa B_1} h_i(x)\le C m_i^{-1+o(1)}, \quad
\max_{B_1}|\nabla h_i|\le C\max_{\bar B_2}h_i \le C m_i^{-1+o(1)} .
\]
Since $u_i$ is nonnegative,  it is easy to check by using \eqref{G} that
\[
|Q_i''(x)|+|\nabla Q_i''(x)|\le Cm_i^{-1+o(1)} \quad \mbox{for all }x\in B_1.
\]

Applying Proposition \ref{prop:pohozaev} to \eqref{eq:flat}, we have
\begin{align}
 \tau_i & \int_{B_1} u_i(x)^{p_i+1}-A_2\int_{B_1}  |x| u_i(x)^{p_i+1}  \,\ud x  \nonumber\\& \leq  C\Big(\int_{B_1}(|Q'_i(x)|+|x||\nabla Q_i'(x)) u_i(x)^{p_i}+ m_i^{-1+o(1)}\int_{ B_1} u_i^{p_i}+  \int_{\pa B_1} u_i^{p_i+1}\,  \ud s\Big).
 \label{eq:a6}
\end{align}
By Proposition \ref{prop:blow up a bubble} and a change of variables,
\begin{align*}
\int_{B_1} u_i(x)^{p_i+1} \,\ud x &\ge C^{-1}\int_{B_{r_i}} \frac{m_i^{p_i+1}}{(1+|m_i^{(p_i-1)/4}y|^2)^{(n-4)(p_i+1)/2}}\,\ud y
\ge C^{-1} m_i^{\tau_i(1-n/4)}.
\end{align*}
By Proposition \ref{prop:blow up a bubble} and  Lemma \ref{lem:upbound1},  we have
\[
 \int_{B_{1}} u_i^{p_i} \le C m_i^{-1+O(\tau_i)}, \quad \int_{\pa B_1} u_i^{p_i+1}\,  \ud s \le Cm_i^{-2n/(n-4) +O(\tau_i)},
\]
\be \label{eq:cor-energy-2}
 \int_{B_{1}} |x|^su_i^{p_i+1} \le C m_i^{-2s/(n-4)+O(\tau_i)} \quad \mbox{for }-n<s<n,
\ee
It follows from Lemma \ref{lem:aux1} that
\[
 \int_{B_{1}} (|Q'_i(x)|+|x||\nabla Q_i'(x)|)  u_i(x)^{p_i+1}\,\ud x\le  Cm_i^{-2/(n-4)+O(\tau_i)}.
\]
The lemma follows after inserting  the above inequalities to \eqref{eq:a6}.
\end{proof}

\begin{lem}\label{lem:aux3} For $-4<s<4$, we have, as $i\to \infty$,
\[
m_i^{1+\frac{2s}{n-4}} \int_{B_{r_i}}|y|^s u_i(y)^{p_i}\,\ud y  \to \int_{\R^n}|z|^s (1+|z|^2)^{-\frac{n+4}{2}}\,\ud z,\]
\[
m_i^{1+\frac{2s}{n-4}} \int_{B_1\setminus B_{r_i}}|y|^s u_i(y)^{p_i}\,\ud y  \to 0.
\]
\end{lem}
\begin{proof} The first assertion  follows from Proposition \ref{prop:blow up a bubble} and Lemma \ref{lem:error} as we computed before.

By Lemma \ref{lem:upbound1} and Lemma \ref{lem:error}, we have
\begin{align*}
\int_{r_i\le |y|<1}|y|^s u_i(y)^{p_i}\,\ud y\le  Cm_i^{-1-\frac{2s}{n-4} } R_i^{s-4}.
\end{align*}
Thus the second convergence result in the lemma follows.
\end{proof}

\begin{prop}\label{prop:upbound2} Under the assumptions in Lemma \ref{lem:upbound1}, we have
\[
u_i(x)\leq Cu_i^{-1}(0)|x|^{4-n}\quad \mbox{for all } |x|\leq 1.
\]
\end{prop}

\begin{proof}  For $|x|\le r_i$, the proposition follows immediately from Proposition \ref{prop:blow up a bubble} and Lemma \ref{lem:error}.

Since $u_i(x)\le A_3 |x|^{-4/(p_i-1)}$ in $B_2$, it follows from Proposition \ref{prop:har} that for any $0<\va<1$ there exists a positive constant
$C(\va)$, depending only on $n, A_1, A_2, A_3$ and $\va$, such that
\be \label{eq:extra1}
\sup_{B_{3/2}\setminus B_\va} u_i \le C(\va)\inf_{B_{3/2}\setminus B_\va} u_i.
\ee

   For any given unit vector $e$,  Let $\varphi_i (x)=u_i(\rho e) ^{-1} u_i(x)$. Since $\varphi_i(\rho e)=1$,  by \eqref{eq:extra1}
\be \label{eq:a8}
\|\varphi_i\|_{L^\infty(B_{3/2}\setminus B_\va)} \le C(\va) \quad \mbox{for }0<\va<1.
\ee By \eqref{eq:s1}, $\varphi_i$ satisfies
\[
\varphi_i(x)=  \int_{B_{3/2}} G_i(x,y)K_i(y) u_i(\rho e)^{p_i-1} \varphi_i(y)^{p_i}\,\ud y+\tilde h_i(x),
\]
where $\tilde h_i(x)=\int_{B_3\setminus B_{3/2}} G_i(x,y)K_i(y) u_i(\rho e)^{p_i-1} \varphi_i(y)^{p_i}\,\ud y +u_i(\rho e) ^{-1} h_i(x)$. By \eqref{G} and \eqref{H},
 \be \label{eq-eq1}
 \sum_{k=1}^ 5 \|\nabla^k\tilde h_i\|_{L^\infty(B_1)} \le C \|\tilde h_i\|_{L^\infty(B_{5/4})} \le C \tilde h_i(\rho e) \le C.
 \ee
By Lemma \ref{lem:upbound1}, $u_i(\rho e) ^{p_i-1} \to 0 $ as $i\to \infty$. Applying Proposition \ref{prop:local estimates} to $\varphi_i$ and making use of \eqref{G}, \eqref{K}, \eqref{eq-eq1} and \eqref{eq:a8}, we have
 \[
 \|\varphi_i\|_{C^4(B_1\setminus B_{\va})}\le C(\va) \quad \mbox{for }0<\va<1.
 \]
 It follows,  after passing to a subsequence, that
 \[
 \tilde h_i \to h  \quad \mbox{in }C^4(B_1), \quad \varphi_i \to \varphi  \quad \mbox{in }C_{loc}^3(B_1\setminus \{0\})
 \]
for some $h\in C^5(B_1)$ and $\varphi\in C^4 (B_1\setminus \{0\})$. Since $ 0$ is an isolated simple blow up point of $\{u_i\}_{i=1}^\infty$, we have  $r^{\frac{n-4}{2}}\bar \varphi(r) \ge \rho^{\frac{n-4}{2}}\bar \varphi(\rho)$ for $0<r<\rho$, where $\bar \varphi(r) =\frac{1}{|\pa B_r|} \int_{\pa B_r}\varphi\,\ud S$. By \eqref{eq:extra1},
  \be \label{eq:-59}
  \bar \varphi (r) \ge \frac{1}{C} r^{-\frac{2-n}{2}}.
  \ee It follows that
\be\label{eq:forceconvergence}
 \int_{B_{3/2}} G_i(x,y)K_i(y) u_i(\rho e)^{p_i-1} \varphi_i(y)^{p_i}\,\ud y= \varphi_i(x)-\tilde h_i(x) \to \varphi(x) -h(x)
\ee
in $C^3_{loc}(B_1\setminus \{0\})$.
For $|x|>0$, let $0< r_i\le  \va<\frac{1}{2}|x|$. By \eqref{eq:a8}, Lemma \ref{lem:error} and Lemma \ref{lem:aux3}, we have
\begin{align*}
 & \int_{B_{1}} G_i(x,y)K_i(y) u_i(\rho e)^{p_i-1} \varphi_i(y)^{p_i}\,\ud y = \int_{B_{\va}} G_i(x,y)K_i(y) u_i(\rho e)^{p_i-1} \varphi_i(y)^{p_i}\,\ud y +o(1) \nonumber\\ &
 =G_i(x,0) u_i(\rho e)^{-1} \int_{B_{r_i}}   u_i(y)^{p_i}\,\ud y + O(m_i)\int_{B_\va\setminus B_{r_i}} u_i(y)^{p_i}\,\ud y  +O(m_i)\int_{B_{\va}}|y| u_i(y)^{p_i}\,\ud y\nonumber \\&
= G_i(x,0) \frac{\int_{\R^n}\left(\frac{1}{1+|y|^2}\right)^\frac{n+4}{2}\ud y+o(1)}{u_i(\rho e) u_i(0)} + o(1).
\end{align*}
By \eqref{G},  after passing to a subsequence, $G_i(x,0)$ converges in $C^5_{loc}(B_1\setminus \{0\})$ to some positive function $G_\infty(x,0)$  as $i\to \infty$.  By \eqref{eq:-59} and \eqref{eq-eq1}, $\varphi(\bar x)-h(\bar x)\ge \frac{1}{C}>0$ at some point $\bar x\in B_1$.   It follows from \eqref{eq:forceconvergence} that
\[
u_i(\rho e) u_i(0) \to \frac{1}{\varphi(\bar x)-h(\bar x)} G_\infty(\bar x,0) \int_{\R^n}\left(\frac{1}{1+|y|^2}\right)^\frac{n+4}{2}\ud y  \quad \mbox{as }i\to \infty.
\]
Therefore, we have
\be \label{eq:a7'}
u_i(\rho e) u_i(0)\le C \quad \forall~|e|=1.
\ee

It follows from Proposition \ref{prop:har} and \eqref{eq:a7'} that Proposition \ref{prop:upbound2} holds for $\rho\le |x| \le 1$.
To establish the inequality in the Proposition for $r_i\le |x|\le \rho$, it suffices to rescale and reduce the problem to the case $|x|=1$.
\end{proof}

\begin{prop}\label{prop:upbound3} Under the assumptions in Lemma \ref{lem:upbound1}, we have, with $k=1,\dots, 4$,
\[
|\nabla^k u_i(x)|\leq Cu_i^{-1}(0)|x|^{4-n-k},\quad \mbox{for all } r_i\le |x|\leq 1.
\]

\end{prop}

\begin{proof} Since $0$ is an isolated blow up point in $B_2$, by Proposition \ref{prop:har} we see that Proposition \ref{prop:upbound2} holds for all $|x|\le \frac32$.  For any $r_i\le |x|<1$, let $
\varphi_i(z)=\Big( \frac{|x|}{4}\Big)^{\frac{4}{p_i-1}} u_i(x+\frac{|x|}{4} z).$
By the equation of $u_i$, we have
\[
\varphi_i(z)=\int_{\{y:|x+\frac{|x|}{4} y|\le 3\}} \tilde G_{i}(z,y) \tilde K_i(y) \varphi_i(y)^{p_i-1} \varphi_i(y)\,\ud y+\tilde h_i(z),
\]
where $\tilde G_{i}(z,y)=(\frac{|x|}{4})^{n-4} G_i(x+\frac{|x|}{4} z, x+\frac{|x|}{4} y)$, $\tilde K_i(y)=K_i(x+\frac{|x|}{4} y)$, and $\tilde h_i(z) =( \frac{|x|}{4})^{\frac{4}{p_i-1}} h_i(x+\frac{|x|}{4} z)$. Since $0$ is an isolated blow up point of $u_i$, we have $\varphi_i(z)^{p_i-1}\le A_2^{p_i-1}$ for all $|z|\le 1$. Since $\varphi_i, \tilde G_i, \tilde K_i$ and $\tilde h_i$ are nonnegative,  by Proposition \ref{prop:local estimates} we have
\[
|\nabla^k \varphi_i(0)|\le C(\|\varphi_i\|_{L^\infty(B_1)} +\| \tilde h_i\|_{C^4(B_1)}).
\]
 This gives
 \begin{align*}
 (\frac{|x|}{4})^k |\nabla^k u_i(x)| \le C\|u_i\|_{L^\infty(B_{\frac{|x|}{4}}(x))}+Cm_i^{-1}
 \le C u_i(0)^{-1}|x|^{4-n}.
 \end{align*}
\end{proof}

\begin{rem} \label{rem:convergence} From the proof of Proposition \ref{prop:upbound2}, it is easy to see that
\[u_i(0) u_i(x)\to aG_\infty(x,0)  +h(x) \quad \mbox{in }C^3_{loc}(B_1\setminus \{0\}),
\]
where \be\label{eq:number a}
a=\int_{\R^n}(1+|y|^2)^{-\frac{n+4}{2}}\,\ud y,
\ee and $h\in C^5(B_1)$ satisfies \eqref{H} with $A_2$ replaced by some $C$ depending only on   $n, A_1, A_2,A_3$ and $\rho$.
\end{rem}

\section{Expansions of blow up solutions of integral equations}
\label{section:bubble-expansion}

 In this section, we are interested in stronger estimates than that in Proposition \ref{prop:upbound2}. To make statements closer to the main goal of the paper, we restrict our attention to a special class of functions $K_i$.  Namely, given $p_i, G_i$, and $ h_i$ satisfying \eqref{p}, \eqref{G} and \eqref{H}  respectively, $\kappa_i$ satisfying \eqref{K} with $K_i$ replaced by $\kappa_i^{\tau_i}$,  let $0\le u_i\in C^4(B_3)$
be a solution of
\be \label{eq:s1'}
u_i(x)=\int_{B_3 }G_i(x,y) \kappa_i(y)^{\tau_i} u_i^{p_i}(y)\,\ud y +h_i(x) \quad \mbox{in  } B_3.
\ee
We also assume that
\be \label{eq:56}
\nabla \kappa_i(0)=0.
\ee Suppose that $0$ is an isolated simple blow up point of $\{u_i\}$ with $\rho=1$, i.e.,
\be\label{eq:A_3}
u_i(x)\leq A_3 |x|^{-4/(p_i-1)}\quad \mbox{for all } x\in B_2.
\ee
and $
r^{4/(p_i-1)} \bar u_i(r)$ has precisely one critical point in $(0,1)$.
Let us first introduce a non-degeneracy result.

\begin{lem}\label{lem:non-degeneracy} For $0<\sigma<n/2$,  let $v\in L_{loc}^\infty(\R^n)$, $v(x)\le C (1+|x|)^t$ for some $t<2\sigma$ and $C>0$,  be a solution of
\[
v(x)=2^{2\sigma} \frac{\Gamma(1+\frac{n}{2}+\sigma)}{\Gamma(1+\frac{n}{2}-\sigma)}N_\sigma \int_{\R^n} \frac{U_{\sigma,1}(y)^{\frac{4\sigma}{n-2\sigma}}v(y)}{|x-y|^{n-2\sigma}}\,\ud y,
\]
where $N_\sigma:= \frac{\Gamma(\frac{n+2\sigma}{2})}{2^{2\sigma}\pi^{\frac{n}{2}} \Gamma(\sigma)}$ and $U_{\sigma,1}=(1+|x|^2)^{-\frac{n-2\sigma}{2}}$.
Then for some constants $a_0,\dots,a_n$,
\[
v(z)=a_0 \left(\frac{n-4}{2}U_{\sigma,1}(z)+z\cdot \nabla U_{\sigma,1}(z)\right)+\sum_{j=1}^n a_j \pa_j U_{\sigma,1}(z).
\]

\end{lem}

\begin{proof} When $0<\sigma<1$, the lemma was proved by D\'avila-del Pino-Sire \cite{DDS}.  When $\sigma\ge 1$, the proof is similar.  Since $v(x)\le C (1+|x|)^t$ for some $t<2\sigma$ and $C>0$, we have, using the integral equation, $|v(x)|\le C (1+|x|)^{2\sigma-n}.$ Let $F:\R^n\to \mathbb{S}^n\setminus\{N\}$, $
F(x)=(\frac{2x}{1+|x|^2}, \frac{1-|x|^2}{1+|x|^2}) $
denote the inverse of the stereographic projection and $ h(F(x)):=v(x) J_{F}(x)^{-\frac{n-2\sigma}{2n}} $, where $J_F=(\frac{2}{1+|x|^2})^{n}$ is the Jacobian determinant of $F$ and $N$ is the north pole. It follows  that $h\in L^\infty(\Sn)$ and satisfies
 \[
L_\sigma h(\xi):= N_\sigma \int_{\Sn} |\xi-\eta|^{2\sigma-n} h(\eta)\,\ud \eta = \frac{\Gamma(1+\frac{n}{2}-\sigma)}{\Gamma(1+\frac{n}{2}+\sigma)} h(\xi).
 \]
 By the regularity theory for Riesz potentials, $h\in C^\infty(\Sn)$.
 By Beckner \cite{Be}, $L_\sigma Y^{(k)}= \frac{\Gamma(k+\frac{n}{2}-\sigma)}{\Gamma(k+\frac{n}{2}+\sigma)} Y^{(k)}$ for every spherical harmonics $Y^{(k)}$ of degree $k$. It follows that $h$ has to be a linear combination of the spherical harmonics of degree one on $\Sn$.  Transforming $h$ back to $\R^n$, we complete the proof.
\end{proof}

In view of Corollary \ref{cor:GM2.8},  we assume  in this and next section that
\be \label{eq:IE-cond}
U_{\lda}(x)=\int_{B_3} G_i(x,y)\{U_\lda(y)^{\frac{n+4}{n-4}}+c_{\lda,i}'(y) U_{\lda}(y)\}\,\ud y+c_{\lda,i}''(x) \quad \forall ~\lda\ge 1 ,~x\in B_3,
\ee
where $c_{\lda,i}',c_{\lda,i}''\in C^5(B_3)$ satisfy
\be \label{eq:non-flat}
\Theta_i:=\sum_{k=0}^5\|\lda^{-k}\nabla^kc_{\lda,i}'\|_{L^\infty(B_2)}\le A_2
,\ee and $\|c_{\lda,i}''\|_{C^5(B_2)}\le A_2 \lda^{\frac{4-n}{2}}$, respectively.

\begin{lem}
\label{lem:expansion-a}
Let $0\le u_i\in C^4(B_3)$ be a solution of \eqref{eq:s1'} and $0$ be an isolated simple blow up point of $\{u_i\}$ with some constant
$\rho$, say $\rho=1$.  Suppose \eqref{eq:IE-cond} holds and let $\Theta_i$ be defined in \eqref{eq:non-flat}. Then we have
\[
|\varphi_i(z)-U_1(z)| \le C \begin{cases}
\max\{\tau_i, m_i^{-2}\},& \quad \mbox{if } 5\le n\le 7, \\
\max\{ \tau_i, \Theta_i m_i^{-2}\log m_i, m_i^{-2}\},& \quad \mbox{if }n=8,\\
\max\{ \tau_i, \Theta_i m_i^{-\frac{8}{n-4}}, m_i^{-2}\},& \quad\mbox{if } n\ge 9,
\end{cases}
 \quad \forall~ |z|\le  m_i^{\frac{p_i-1}{4}},
\]where $\varphi_i(z)=\frac{1}{m_i}u_i(m_i^{-\frac{p_i-1}{4}}z)$, $m_i=u_i(0)$, and $C>0$ depends only on $n,A_1,A_2$ and $A_3$.
\end{lem}

\begin{proof}
For brevity, set $\ell_i= m_i^{\frac{p_i-1}{4}}$. By the equation satisfied by $u_i$, we have
\be \label{eq:varphi}
\varphi_i(z)=\int_{B_{\ell_i}} G_{i,\ell_i^{-1}}(z,y)\tilde \kappa_i(y)^{\tau_i}\varphi_i(y)^{p_i}\,\ud y +\bar h_i(z),
\ee
where $G_{i,\ell_i^{-1}}(z,y)=\ell_i^{4-n}G_i(\ell_i^{-1}x,\ell_i^{-1} y)$, $\tilde \kappa_i(z)=\kappa_i(\ell_i^{-1} z)$, and
$
\bar h_i(z)=m_i^{-1}\tilde h_i(\ell_i^{-1}z)
$ with
\[
\tilde h_i (x)= \int_{B_3\setminus B_1} G_i(x,y) \kappa_i(y)^{\tau_i} u_i(y)^{p_i} \,\ud y+h_i(x).
\]
Since $0$ is an isolated simple blow up point of $u_i$, by Proposition \ref{prop:upbound2} we have
\be \label{eq:exp-1}
u_i(x) \le C m_i^{-1}|x|^{4-n} \quad \mbox{for }|x|<1.
\ee It follows that $\tilde h_i(x)\le C m_i^{-1}$ for $x\in B_1$ and $\bar h_i(z) \le Cm_i^{-2}$ for $z\in B_{\ell_i}$.

Notice that $U_{\ell_i}(x)\le Cm_i^{-1}$ for $1\le |x|\le 3$.  Let $z=\ell_i x$. By \eqref{eq:IE-cond} with $\lda=\ell_i$ we have for $|z|\le \ell_i$
\begin{align}
U_1(z)&=\int_{B_{\ell_i}} G_{i,\ell_i^{-1}}(z,y)(U_1(y)^{\frac{n+4}{n-4}}+m_i^{-\frac{8}{n-4}}c'_{\ell_i,i}(\ell_i^{-1}y)U_1(y))\,\ud y+\mathcal{O}(m_{i}^{-2})\nonumber \\&
=\int_{B_{\ell_i}} G_{i,\ell_i^{-1}}(z,y) (\tilde \kappa_i(y)^{\tau_i}U_1(y)^{p_i}+T_i(y))\,\ud y+\mathcal{O}(m_i^{-2}),
\label{eq:U1}
\end{align}
where we used $m_i^{\tau_i}=1+o(1)$, and
\be \label{eq:T-i}
T_i(y):=U_1(y)^{\frac{n+4}{n-4}}-\tilde \kappa_i(y)^{\tau_i}U_1(y)^{p_i} +m_i^{-\frac{8}{n-4}}c'_{\ell_i,i}(\ell_i^{-1} y)U_1(y).
\ee
\emph{Here and throughout this section, $\mathcal{O}(m_i^{-2})$ denotes some function $f_i$ satisfying $\|\nabla ^k f_i\|_{B_{(1-\va)\ell_i}} \le C(\va) m_i^{-2-\frac{2k}{n-4}}$ for small $\va>0$ and $k=0,\dots, 5$.}

In the following, we adapt some arguments from Marques \cite{Marques} for the Yamabe equation; see also the proof of Proposition 2.2 of Li-Zhang \cite{Li-Zhang05}.
 Let
\[
\Lda_i=\max_{|z|\le \ell_i} |\varphi_i-U_1|.
\]
By \eqref{eq:exp-1}, for any $0<\va<1$ and $\va \ell_i\le |z|\le \ell_i$, we have $|\varphi_i (z)-U_1(z)|\le C(\va) m_i^{-2}$, where we used $m_i^{\tau_i}=1+o(1)$. Hence, we may assume that $\Lda_i$ is achieved at some point $|z_i|\le \frac12 \ell_i$, otherwise the proof is finished.
Set
\[
v_i(z)= \frac{1}{\Lda_i}(\varphi_i(z)-U_1(z)).
\]
It follows from \eqref{eq:varphi} and \eqref{eq:U1} that $v_i$ satisfies
\be \label{eq:newscale1}
v_i(z)= \int_{B_{\ell_i}} G_{i,\ell_i^{-1}}(z,y) (b_i(y) v_i(y)+\frac{1}{\Lda_i}T_i(y))\,\ud y+\frac{1}{\Lda_i}\mathcal{O}(m_i^{-2}),
\ee
where
\be \label{eq:b-i}
b_i=\tilde \kappa_i^{\tau_i}\frac{\varphi_i^{p_i}-U_1^{p_i}}{\varphi_i-U_1}.
\ee
Since
\[
G_{i,\ell_i^{-1}}(z,y)\le A_1|z-y|^{4-n}
\] and
\be \label{eq:Ti-estimate}
|T_i(y)|\le C\tau_i (|\log  U_i|+|\log \tilde \kappa_i|)(1+|y|)^{-p_i(n-4)}+\Theta_i m_i^{-\frac{8}{n-4}} (1+|y|)^{4-n},
\ee we obtain
\[
\int_{B_{\ell_i}} G_{i,\ell_i^{-1}}(z,y) |T_i(y)|\,\ud y\le C (\tau_i+\Theta_i \al_i)\quad \mbox{for }|z|\le \frac{\ell_i}{2},
\]
where
\be \label{eq:ali}
\al_i= \begin{cases}
m_i^{-2},& \quad \mbox{if }5\le n\le 7, \\
m_i^{-2}\log m_i,& \quad \mbox{if }n=8,\\
m_i^{-\frac{8}{n-4}},& \quad \mbox{if }n\ge 9.
\end{cases}
\ee
Since $\kappa_i(x)$ is bounded and $\varphi_i\le C U_1$, we see that
 \be\label{eq:bi-estimate}
 |b_i(y)|\le CU_1(y)^{p_i-1}\le C (1+|y|)^{-7.5}, \quad y\in B_{\ell_i}.
 \ee
Noticing that $\|v_i\|_{L^\infty(B_{\ell_i})}\le 1$,    we have
\[
\int_{B_{\ell_i}} G_{i,\ell_i^{-1}}(z,y) |b_i(y) v_i(y)|\,\ud y\le C(1+|z|)^{-\min\{n-4,3.5\}}.
\] Hence, we obtain
\be\label{eq:b-expansion-1}
v_i (z) \le C((1+|z|)^{-\min\{n-4,3.5\}}+\frac{1}{\Lda_i}(\tau_i+\Theta_i\al_i+m_i^{-2})) \quad \mbox{for }|z|\le \frac{\ell_i}{2}.
\ee

Suppose by contradiction that $\frac{1}{\Lda_i}\max\{\tau_i,\Theta_i\al_i,m_i^{-2} \}\to 0$ as $i \to \infty$. Since $v(z_i)=1$, by \eqref{eq:b-expansion-1} we see that $
|z_i|\le C.$
Differentiating the integral equation \eqref{eq:newscale1} up to three times, together with \eqref{eq:Ti-estimate} and \eqref{eq:bi-estimate}, we see that the $C^3$ norm of $v_i$ on any compact set  is uniformly bounded. By the Arzel\`a-Ascoli theorem let $v:=\lim_{i\to \infty}v_i$ after passing to a subsequence. Using Lebesgue's dominated convergence theorem, we obtain
\[
v(z)=c_{n}\int_{\R^n}\frac{U_1(y)^{\frac{8}{n-4}} v(y)}{|z-y|^{n-4}}\,\ud y.
\]
It follows from Lemma \ref{lem:non-degeneracy} that
\[
v(z)=a_0 (\frac{n-4}{2}U_1(z)+z\cdot \nabla U_1(z))+\sum_{j=1}^n a_j \pa_j U_1(z),
\]
where $a_0,\dots,a_n$ are constants.
Since $v(0)=0$ and $\nabla v(0)=0$, $v$ has to be zero. However, $v(z_i)=1$. We obtain a contradiction. Therefore, $\Lda_i\le C (\tau_i+\al_i)$ and the proof is completed.
\end{proof}

\begin{lem}\label{lem:expansion-b} Under the same assumptions as in Lemma \ref{lem:expansion-a}, we have
\[
\tau_i \le C \begin{cases}
m_i^{-2},& \quad \mbox{if }5\le n\le 7, \\
\max\{ \Theta_i m_i^{-2}\log m_i,m_i^{-2}\},& \quad \mbox{if }n=8,\\
\max\{ \Theta_i  m_i^{-\frac{8}{n-4}},m_i^{-2}\},& \quad \mbox{if }n\ge 9.
\end{cases}
\]
\end{lem}

\begin{proof} The proof is also by contradiction. Recall the definition of $\al_i$ in \eqref{eq:ali}. Suppose by contradiction that $\frac{1}{\tau_i}\max\{\Theta_i\al_i,m_i^{-2} \}\to 0$ as $i\to \infty$.
Set
\[
v_i(z)=\frac{\varphi_i(z)-U_1(z)}{\tau_i}.
\]
It follows from Lemma \ref{lem:expansion-a} that $|v_i(z)|\le C$ in $B_{\ell_i}$, where $\ell_i= m_i^{\frac{p_i-1}{4}}$. As \eqref{eq:newscale1}, we have
\be \label{eq:newscale2}
v_i(z)= \int_{B_{\ell_i}} G_{i,\ell_i^{-1}}(z,y) (b_i(y) v_i(y)+\frac{1}{\tau_i}T_i(y))\,\ud y+\frac{1}{\tau_i}\mathcal{O}(m_i^{-2}),
\ee
where $b_i$ and $T_i$ are given by \eqref{eq:b-i} and \eqref{eq:T-i}, respectively.

By the estimates \eqref{eq:bi-estimate} and \eqref{eq:Ti-estimate} for $b_i$ and $T_i$ respectively, we conclude from the integral equation that $\|v_i\|_{C^3}$ is uniformly bounded over any compact set. It follows that $v_i\to v$ in $C^2_{loc}(\R^n)$ for some $v\in C^3(\R^n)$.

Multiplying both sides of \eqref{eq:newscale2} by $b_i(z)\phi(z)$, where $\phi(z)=\frac{n-4}{2} U_1(z)+z\cdot \nabla U_1(z)$, and integrating over $B_{\ell_i}$, we have, using the symmetry of $G_{i,\ell_i^{-1}}$ in $y$ and $z$,
\begin{align*}
&\int_{B_{\ell_i}}b_i(z)v_i(z)\left(\phi(z)-\int_{B_{\ell_i}} G_{i,\ell_i^{-1}}(z,y) b_i(y) \phi(y)\,\ud y\right)\,\ud z\\&
=\frac{1}{\tau_i}\int_{B_{\ell_i}}T_i(z)\int_{B_{\ell_i}} G_{i,\ell_i^{-1}}(z,y) b_i(y) \phi(y)\,\ud y\,\ud z+\frac{1}{\tau_i}\mathcal{O}(m_i^{-2})\int_{B_{\ell_i}} b_i(z)\phi(z)\,\ud z.
\end{align*}
As $i\to \infty$, we have
\[
\int_{B_{\ell_i}} G_{i,\ell_i^{-1}}(z,y) b_i(y) \phi(y)\,\ud y \to c_{n}\int_{\R^n} \frac{U_1(y)^{\frac{8}{n-4}}\phi(y)}{|z-y|^{n-4}}\,\ud z=\phi(z),
\]
\[
\frac{1}{\tau_i}\mathcal{O}(m_i^{-2})\int_{B_{\ell_i}} b_i(z)\phi(z)\,\ud z \to 0\mbox{ by the contradiction hypothesis},
\]
and
\[
\frac{T_i(z)}{\tau_i}  \to (\log U_1(z))U_1(z)^{\frac{n+4}{n-4}}.
\]
Hence, by Lebesgue's dominated convergence theorem we obtain
\[
\begin{split}
\lim_{i\to \infty}&\frac{1}{\tau_i}\int_{B_{\ell_i}}T_i(z)\int_{B_{\ell_i}} G_{i,\ell_i^{-1}}(z,y) b_i(y) \phi(y)\,\ud y\,\ud z=\int_{\R^n} \phi(z) (\log U_1(z))U_1(z)^{\frac{n+4}{n-4}}\ud z=0.
\end{split}
\]
This is impossible, because
\begin{align*}
\int_{\R^n} \phi(z) (\log U_1(z))U_1(z)^{\frac{n+4}{n-4}}\ud z&=\frac{(n-4)^2|\mathbb{S}^{n-1}|}{4} \int_{0}^\infty  \frac{(r^2-1)r^{n-1}}{(1+r^2)^{n+1}} \log (1+r^2)\,\ud r\\&
=\frac{(n-4)^2|\mathbb{S}^{n-1}|}{2} \int_{1}^\infty  \frac{(r^2-1)r^{n-1}}{(1+r^2)^{n+1}} \log r\,\ud r>0,
\end{align*} where we used
\[
\int_{0}^1\frac{(r^2-1)r^{n-1}}{(1+r^2)^{n+1}} \log (1+r^2)\,\ud r= -\int_{1}^\infty \frac{(s^2-1)s^{n-1}}{(1+s^2)^{n+1}} (\log (1+s^2)-\log s^2)\,\ud s
\]
by the change of variable $r=\frac{1}{s}$.

We obtain a contradiction and thus $\tau_i\le \al_i$. Therefore, the lemma is proved.
\end{proof}

\begin{prop}\label{prop:expansion} Under the hypotheses in Lemma \ref{lem:expansion-a},  we have
\[
|\varphi_i (z)-U_1(z)| \le C \begin{cases}
m_i^{-2},& \quad \mbox{if } 5\le n\le 7, \\
\max\{ \Theta_i m_i^{-2}\log m_i,m_i^{-2}\},& \quad \mbox{if }n=8,\\
\max\{ \Theta_i  m_i^{-\frac{8}{n-4}},m_i^{-2}\},& \quad\mbox{if } n\ge 9,
\end{cases}
 \quad \forall~ |z|\le  m_i^{\frac{p_i-1}{4}}.
\]
\end{prop}

\begin{proof} It follows immediately from Lemma \ref{lem:expansion-a} and  Lemma \ref{lem:expansion-b}.

\end{proof}

\begin{prop}\label{prop:expansion8+} Under the hypotheses in Lemma \ref{lem:expansion-a}, we have, for every $|z|\le  m_i^{\frac{p_i-1}{4}}$,
\[
|\varphi_i (z)-U_1(z)| \le C \begin{cases}
\max\{ \Theta_im_i^{-2}m_i^{\frac{2}{n-4}}(1+|z|)^{-1}, m_i^{-2}\},& \quad \mbox{if } n=8, \\
\max\{ \Theta_i m_i^{-2}m_i^{\frac{2(n-8)}{n-4}}(1+|z|)^{8-n}, m_i^{-2}\} ,& \quad\mbox{if } n\ge 9.
\end{cases}
\]
\end{prop}
\begin{proof} Let $\al_i$ be defined in \eqref{eq:ali}. We may assume that  $\frac{m_i^{-2}}{\Theta_i\al_i}\to 0$ as $i\to \infty$ for $n\ge 8$; otherwise the proposition follows immediately from Proposition \ref{prop:expansion}.  Set
\[
\al'_i=\begin{cases}
m_i^{-2}m_i^{\frac{2}{n-4}},& \quad \mbox{if } n=8, \\
m_i^{-2}m_i^{\frac{2(n-8)}{n-4}},& \quad\mbox{if } n\ge 9,
\end{cases}
\]
and
\[
v_i(z)=\frac{\varphi_i(z)-U_1(z)}{\Theta_i \al'_i}, \quad |z|\le  m_i^{\frac{p_i-1}{4}}.
\]
Since $\frac{m_i^{-2}}{\Theta_i\al_i}\to 0$, it follows from Proposition \ref{prop:expansion} that $|v_i|\le C$. Since $0<\varphi_i\le CU_1$, we only need to prove the proposition when $|z|\le \frac{1}{2}\ell_i$, where  $\ell_i=m_i^{\frac{p_i-1}{4}}$.  Similar to \eqref{eq:newscale1}, $v_i$ now satisfies
\[
v_i(z)=\int_{B_{\ell_i}} G_{i,\ell_i^{-1}}(z,y) (b_i(y) v_i(y)+\frac{1}{\Theta_i \al'_i}T_i(y))\,\ud y+\frac{1}{\Theta_i \al_i'}\mathcal{O}(m_i^{-2}),
\]
where $b_i$ and $T_i$ are given by \eqref{eq:b-i} and \eqref{eq:T-i}, respectively.
Noticing that
\[
|T_i(y)|\le C\tau_i (|\log  U_1|+|\log \tilde \kappa_i|)(1+|y|)^{-4-n}+m_i^{-\frac{8}{n-4}} \Theta_i (1+|y|)^{4-n},
\]
we have
\begin{align*}
\frac{1}{\Theta_i \al_i'}\int_{B_{\ell_i}} G_{i,\ell_i^{-1}}(z,y) |T_i(y)|\,\ud y&\le C\int_{B_{\ell_i}} \frac{1}{|z-y|^{n-4}(1+|y|)^4 m_i^{\frac{2}{n-4}}}  \,\ud y\\&
\le C \int_{B_{\ell_i}} \frac{1}{|z-y|^{n-4}(1+|y|)^5}  \,\ud y
\le C(1+|z|)^{-1} \quad \mbox{for }n=8,
\end{align*}
\[
\frac{1}{\Theta_i\al_i'}\int_{B_{\ell_i}} G_{i,\ell_i^{-1}}(z,y) |T_i(y)|\,\ud y \le C(1+|z|)^{8-n} \quad \mbox{for }n\ge 9.
\]
Thus
\[
|v_i(z)|\le C((1+|z|)^{-3.5}+(1+|z|)^{-1}) \quad \mbox{for }n=8,
\]
\[
|v_i(z)|\le C((1+|z|)^{-3.5}+(1+|z|)^{8-n}) \quad \mbox{for }n\ge 9.
\] If $n=8,9,10,11$,  the conclusion follows immediately from multiplying both sides of the above inequalities by $\al'_i$. If $n\ge 12$, the above estimate gives $|v_i(z)|\le C(1+|z|)^{-3.5}$. Plugging this estimate to the term $\int G_{i,\ell_i^{-1}}(z,y)b_i(y)v_i(y)\,\ud y$ yields $|v_i(z)|\le C(1+|z|)^{8-n}$ as long as $n\le 14$. Repeating this process, we complete the proof.

\end{proof}

\begin{cor}\label{cor:hot-expansion} Under the hypotheses in Lemma \ref{lem:expansion-a}, we have, for $|z|\le  m_i^{\frac{p_i-1}{4}}$, $k=1,2,3,4$,
\begin{align*}
&|\nabla^k(\varphi_i-U_1)(z)|\\&  \le C(1+|z|)^{-k} \begin{cases}
m_i^{-2},& \quad \mbox{if } 5\le n\le 7,\\
\max\{ \Theta_im_i^{-2}m_i^{\frac{2}{n-4}}(1+|z|)^{-1}, m_i^{-2}\},& \quad \mbox{if } n=8, \\
\max\{ \Theta_i m_i^{-2}m_i^{\frac{2(n-8)}{n-4}}(1+|z|)^{8-n}, m_i^{-2}\} ,& \quad\mbox{if } n\ge 9.
\end{cases}
\end{align*}
\end{cor}

\begin{proof} Considering the integral equation of $v_i=\varphi_i-U_1$, the conclusion follows immediately.
Indeed, if $k<4$, we can differentiate the integral equation for $v_i$ directly. If $k=4$, we can use a standard  technique (see the proof of Proposition \ref{prop:local estimates}) for proving the higher order regularity of Riesz potential since $v_i$ and the coefficients are of $C^1$.
\end{proof}

\section{Blow up local solutions of fourth order equations}

\label{section:Q equation blow}

In the previous two sections, we have analyzed the blow up profiles of the blow up local solutions of integral equations. In this section, we will assume that those blow up solutions also satisfy differential equations, which is only used to check the Pohozaev identity in Proposition \ref{prop:4-pohozaev}. It should be possible to completely avoid using differential equations after improving Corollary \ref{cor:Q-gf-expansion}. This is the case on the sphere; see our joint work with Jin  \cite{JLX3}. On the other hand, as mentioned in the Introduction, without additional information fourth order differential equations themselves are not enough to do blow up analysis for positive local solutions.

\begin{prop}\label{prop:one-side} In addition to the hypotheses in Lemma \ref{lem:expansion-a}, assume that $u_i$ also satisfies
\be \label{eq:Q-sub}
P_{ g_i} u_i=c(n) \kappa_i^{\tau_i} u_i^{p_i}\quad \mbox{in }B_3,
\ee
where $\det g_i=1$, $B_3$ is a normal coordinates chart of $g_i$ at $0$ and $\|g_i\|_{C^{10}(B_3)}\le A_1$. Let
\[
 \mathcal{G}_i:= \sum_{k\ge 1,~  2\le k+l\le 4}\Theta_i \|\nabla^k g_i\|_{L^\infty(B_3)}^l+\sum_{k\ge 1,~  6\le k+l\le 8}\|\nabla^k g_i\|_{L^\infty(B_3)}^l ,
 \]
\[
\beta_i:=\begin{cases} (\log m_i)^{-1},& \quad \mbox{if }n=8,\\
m_i^{-\frac{2}{n-4}}, &\quad \mbox{if }n=9,\\
m_i^{-\frac{4}{n-4}}\log m_i, &\quad \mbox{if }n=10,\\
m_i^{-\frac{4}{n-4}}, &\quad \mbox{if }n\ge 11.
\end{cases}
\]
Then, after passing to a subsequence, there exists $C^*>0$ depending only on $n,A_1,A_2, A_3$ such that
\begin{itemize}
\item[(i)] If either $n\le 9$ or $g_i$ is flat, we have, with $\Gamma(\cdot)=\lim_{i\to \infty} u_i(0)u_i(\cdot)$ as in Remark \ref{rem:convergence},
\be \label{eq:sign restrict}
\liminf_{r\to 0}\mathcal{P}(r,\Gamma)\ge 0.
\ee
\item[(ii)]  \eqref{eq:sign restrict} still holds if $n\ge 10$ and
\be \label{eq:r2}
|W_{g_i}(0)|^2 > C^*  \mathcal{G}_i \beta_i.
 \ee

\item[(iii)]
\[
|W_{g_i}(0)|^2 \le C^* \beta_i \begin{cases}
\mathcal{G}_i+1,  & if ~n=8,9,\\
\mathcal{G}_i +(\log m_i)^{-1}, & if ~ n=10, \\
\mathcal{G}_i + m_i^{-\frac{2(n-10)}{n-4}}, &  if ~ n\ge 11 .
\end{cases}
\]
 \end{itemize}
\end{prop}

\begin{proof}
Notice that for every $0<r<1$
\[
m_i^{2}\mathcal{P}(r,u_i)\to \mathcal{P}(r,\Gamma)\quad \mbox{as }i\to \infty.
\]
By Proposition \ref{prop:4-pohozaev}, \eqref{eq:poho-2.2} holds with $u=u_i$,
\begin{align*}
E(u_i):&=P_{g_i}u_i-\Delta^2 u_i\\&
=\frac{n-4}{2}Q_{g_i}u_i+ f^{(1)}_{i,k}\pa_ku_i +f^{(2)}_{i,kl}\pa_{kl}u_i+f^{(3)}_{i,kls}\pa_{kls}u_i+f^{(4)}_{i,klst}\pa_{klst}u_i,
\end{align*}
\be \label{eq:fi}
| f^{(1)}_{i,k}(x)|+ | f^{(2)}_{i,kl}(x)|+ ||x|^{-1}  f^{(3)}_{i,kls}(x)| +| |x|^{-2} f^{(4)}_{i,klst}(x)| \le C\sum_{k\ge 1, 2\le k+1\le 4} \|\nabla^k g_i\|_{L^\infty(B_\delta)}^l,
\ee
\begin{align*}
\mathcal{N}(r,u_i)=\frac{c(n)\tau_i}{p_i+1} \int_{B_r} (\frac{n-4}{2}\kappa_i^{\tau_i}+ x^k\pa_k \kappa_i \kappa_i^{\tau_i-1}) u_i^{p_i+1}
-\frac{r}{p_i+1}\int_{\pa B_{r}} c(n)\kappa_i^{\tau_i} u_i^{p_i+1}.
\end{align*}
By Proposition \ref{prop:upbound2}, for $0<r<1$ we have, for some $C>0$ independent of $i$ and $r$,
\begin{align}\label{eq-eq2}
m_i^{2}\mathcal{N}(r,u_i) \ge
-\frac{m_i^{2}r}{p_i+1}\int_{\pa B_{r}} c(n)\kappa_i^{\tau_i} u_i^{p_i+1}\ge -Cr^{-n}m_i^{1-p_i},
\end{align}
where we used the facts that $\kappa_i(x)=1+O(|x|^2)$ and $|\nabla \kappa_i(x)|=O(|x|)$ with $O(\cdot)$ independent of $i$. Hence, we have
\be \label{eq:tos-1}
\liminf_{i\to \infty}m_i^{2}\mathcal{N}(r,u_i)\ge 0.
\ee

Throughout this section, unless  otherwise stated, we use $C$ to denote some constants independent of $i$ and $r$.

If $g_i$ is flat, we complete the proof because $E(u_i)=0$.

Now  we assume $g_i$ is not flat. By a change of variables $z=\ell_i x$ with $\ell_i=m_i^{\frac{p_i-1}{4}}$, we have
\begin{align*}
\mathcal{ E}_i(r):&= m_i^{2}\int_{B_r} (x^k \pa_k u_i +\frac{n-4}{2} u_i) E(u_i) \,\ud x\\&
=m_i^{2} m_i^{2+(4-n)\frac{p_i-1}{4}} \int_{B_{\ell_i r}}\Big( z^k\pa_k\varphi_i+\frac{n-4}{2}\varphi_i\Big)\cdot\\& \quad
\Big(\frac{n-4}{2}\ell_i^{-4}Q_{g_i}(\ell_i^{-1}z)\varphi_i+ \sum_{j=1}^4\ell_i^{-4+j}f_i^{(j)}(\ell_i^{-1}z)\nabla^j \varphi_i\Big)\,\ud z,
\end{align*}
where $\varphi_i(z)=m_i^{-1}u_i(m_i^{-\frac{p_i-1}{4}}z)$, $
f_i^{(1)}(\ell_i^{-1}z)\nabla^1 \varphi_i=f^{(1)}_{i,k}(\ell_i^{-1}z)\pa_k\varphi_i$
and $f_i^{(j)}(\ell_i^{-1}z)\nabla^j \varphi_i$ is defined in the same fashion for $j\neq 1$. Define
\begin{align*}
\mathcal{\hat E}_i(r):&=m_i^{2} m_i^{2+(4-n)\frac{p_i-1}{4}} \int_{B_{\ell_i r}}\Big( z^k\pa_kU_1+\frac{n-4}{2}U_1\Big)\cdot\\& \quad
\Big(\frac{n-4}{2}\ell_i^{-4}Q_{g_i}(\ell_i^{-1}z)U_1+ \sum_{j=1}^4\ell_i^{-4+j}f_i^{(j)}(\ell_i^{-1}z)\nabla^j U_1\Big)\,\ud z.
\end{align*}
Notice that $m_i^{2+(4-n)\frac{p_i-1}{4}}=1+o(1)$, and $Q_{\tilde g}=O(1)$. By Proposition \ref{prop:expansion}, Proposition \ref{prop:expansion8+}, Corollary \ref{cor:hot-expansion}, \eqref{eq:fi}, we have
\begin{align}
&|\mathcal{ E}_i(r)-\mathcal{\hat E}_i(r)|\nonumber \\&\le C  \sum_{k\ge 1,~  2\le k+l\le 4}\|\nabla^k g_i\|_{L^\infty(B_3)}^l m_i^{2}m_i^{-\frac{4}{n-4}}\int_{B_{\ell_i r}} \sum_{j=0}^4|\nabla^j(\varphi_i-U_1)|(z)(1+|z|)^{2-n+j}\,\ud z\nonumber \\&
\le C \sum_{k\ge 1,~  2\le k+l\le 4}\|\nabla^k g_i\|_{L^\infty(B_3)}^l
\begin{cases}
r^2, & \quad \mbox{if }n=5,6,7,\\
\Theta_i r+r^2, & \quad \mbox{if }n=8,9,\\
\Theta_i  \log (rm_i)+  r^2,& \quad \mbox{if }n=10,\\
\Theta_i m_i^{\frac{2(n-10)}{n-4}}+ r^2,& \quad \mbox{if }n\ge 11.
\end{cases}
\label{eq:tos-2}
\end{align}

Now we estimate $\mathcal{\hat E}_i(r)$.

If $n=5,6,7$, we have
\[
\begin{split}
\mathcal{\hat E}_i(r)&= m_i^{2+(n-4)\tau_i}\int_{B_r} (x^k \pa_k U_{\ell_i} +\frac{n-4}{2} U_{\ell_i}) E(U_{\ell_i}) \,\ud x\\
&= m_i^{2+(n-4)\tau_i}\int_{B_r} (x^k \pa_k U_{\ell_i} +\frac{n-4}{2} U_{\ell_i}) (P_{ g_i}-\Delta^2)U_{\ell_i}\,\ud x\\&
=O(1)m_i^{2}\int_{B_r} |x^k \pa_k U_{\ell_i}+\frac{n-4}{2} U_{\ell_i}| U_{\ell_i}\,\ud x,
\end{split}
\]
where we have  used the estimate $(P_{ g_i}-\Delta^2)U_{\ell_i}=O(1) U_{\ell_i}$ because of \eqref{eq:cor-GM}, and
\[
||x|^k\nabla_x^k U_{\ell_i}(x)|\le C(n,k) U_{\ell_i}(x)\quad  \mbox{for } k\in \mathbb{N}.
\]
Hence,
\be \label{eq:tos-3}
|\mathcal{\hat E}_i(r)|\le Cr^{8-n}\le Cr.
\ee
Therefore, \eqref{eq:sign restrict} follows from \eqref{eq:tos-1}, \eqref{eq:tos-2} and \eqref{eq:tos-3} when $n=5,6,7$.

If $n\ge 8$,  by Lemma \ref{lem:GM2.8} we have
\begin{align*}
\mathcal{\hat E}_i(r) =& -\frac{2}{n}\gamma_i\int_{B_{\ell_i r}}(s\pa_s U_1+\frac{n-4}{2}U_1)(c_1^*s\pa_s U_1+c_2^*s^2 \pa_{ss}U_1)\,\ud z\\&
-\frac{32(n-1)\gamma_i}{3(n-2)n^2} \int_{B_{\ell_i r}}(s\pa_s U_1+\frac{n-4}{2}U_1)s^2(\pa_{ss}U_1-\frac{\pa_s U_1}{s})\,\ud z \\&
+(n-4)\gamma_i\int_{B_{\ell_i r}}(s\pa_s U_1+\frac{n-4}{2}U_1)U_1\,\ud z+O(\al_i'')\sum_{k\ge 1,~  6\le k+l\le 8}\|\nabla^k g_i\|_{L^\infty(B_3)}^l,
\end{align*}
where we have used the cancellation that those terms involving homogeneous polynomials of odd degrees are canceled, $s=|z|$, $\gamma_i=\frac{m_i^{\frac{2(n-8)}{n-4}+(n-4)\tau_i}|W_{g_i}(0)|^2}{24(n-1)}\ge 0$,
\[
\al_i''=\int_{B_r} |x|^2 U_{\ell_i}(x)^2\,\ud x=O(1)\begin{cases}
r^{10-n}, & \quad \mbox{if }n=8,9,\\
\log rm_i,& \quad \mbox{if }n=10,\\
m_i^{\frac{2(n-10)}{n-4}},& \quad \mbox{if }n\ge 11,
\end{cases}
\]
and $c_1^*, c_2^*$ are given in Lemma \ref{lem:GM2.8}.
By direct computations,
\[
r\pa_r U_1+\frac{n-4}{2}U_1=\frac{n-4}{2}\frac{1-r^2}{(1+r^2)^{\frac{n-2}{2}}},
\]
\begin{align*}
c_1^*r\pa_r U_1+c_2^*r^2 \pa_{rr}U_1&=(4-n)\frac{(c_1^*+c_2^*)r^2}{(1+r^2)^{\frac{n-2}{2}}}+(4-n)(2-n)\frac{c_2^*r^4}{(1+r^2)^{\frac{n}{2}}}\\
&=(4-n)\frac{(c_1^*+c_2^*)r^2+(c_1^*+(3-n)c_2^*)r^4}{(1+r^2)^{\frac{n}{2}}},
\end{align*}
\[
\pa_{rr} U_1-\frac{\pa_r U_1}{r}=(n-4)(n-2)(1+r^2)^\frac{-n}{2} r^2.
\]
Thus
\begin{align*}
\mathcal{\hat E}_i(r) = \frac{(n-4)^2}{n}\gamma_i |\mathbb{S}^{n-1}| J_i+O(\al_i'')\sum_{k\ge 1,~  6\le k+l\le 8}\|\nabla^k g_i\|_{L^\infty(B_3)}^l,
\end{align*}
where
\[
J_i:=\int^{\ell_i r}_0 \frac{(1-s^2)[\frac{n}{2}+(c_1^*+c_2^*+n)s^2+(c_1^*+(3-n)c_2^*-\frac{16(n-1)}{3n}+\frac{n}{2})s^4]s^{n-1}}{(1+s^2)^{n-1}}\,\ud s.
\]
If $n=8$, we have $-(c_1^*+(3-n)c_2^*+\frac{n}{2})= (2n-12)+\frac{14}{3}-4=\frac{14}{3}$. Since $\int^{\ell_i r}_0 \frac{s^{13}}{(1+s^2)^7}\,\ud s\to \infty$ as $\ell_i\to \infty$, Hence, $J_i\to \infty$ as $i\to \infty$.
For $n\ge 9$, we  notice that for positive integers $2< m+1<2k$,
\[
\int_0^\infty \frac{t^m}{(1+t^2)^k}\,\ud t=\frac{m-1}{2k-m-1}\int_{0}^\infty \frac{t^{m-2}}{(1+t^2)^k}\,\ud t.
\]
If $\ell_i r =\infty$,  we have
\begin{align*}
J_i=&\Big\{-\frac{2n}{n-4}-(c_1^*+c_2^*+n)\frac{8n}{(n-6)(n-4)}\\&-(c_1^*+(3-n)c_2^*-\frac{16(n-1)}{3n}+\frac{n}{2}) \frac{12n(n+2)}{(n-8)(n-6)(n-4)}\Big\}  \int_0^\infty \frac{s^{n-1}}{(1+s^2)^{n-1}}\,\ud s.
\end{align*}
We compute  the coefficients  of the integral,  \begin{align*}
&-\frac{2n}{n-4}-(c_1^*+c_2^*+n)\frac{8n}{(n-6)(n-4)}\\&\quad -(c_1^*+(3-n)c_2^*-\frac{16(n-1)}{3n}+\frac{n}{2}) \frac{12n(n+2)}{(n-8)(n-6)(n-4)}\\&
=\frac{2n}{n-4}\Big\{-1+(\frac{n(n-2)}{2}-8)\frac{4}{(n-6)}+(\frac{3n}{2}+\frac{16(n-1)}{3n}-12) \frac{6(n+2)}{(n-8)(n-6)}\Big\}\\&
= \frac{2n}{n-4}\Big\{\frac{2n^2 -5n-26}{n-6}+\frac{9(n+2)}{(n-6)}+\frac{32(n-1)(n+2)}{n(n-8)(n-6)}\Big\}\\&
\ge \frac{4n(n^2+2n-4)}{(n-4)(n-6)}>0.
\end{align*}
Therefore, for any $0<r<1$ and sufficiently large $i$ (the largeness of $i$ may depend on $r$), we have
\[
J_i\ge 1/C(n)>0.
\] In conclusion, we obtain
\be \label{eq:tos-6}
\mathcal{\hat E}_i(r) \ge \begin{cases} \frac{1}{C}|W_{g_i}(0)|^2 \log m_i-O(\al''_i)\sum_{k\ge 1,~  6\le k+l\le 8}\|\nabla^k g_i\|_{L^\infty(B_3)}^l, &\quad  \mbox{if }n=8,\\
 \frac{1}{C} |W_{g_i}(0)|^2 m_i^{\frac{2(n-8)}{n-4}}-O(\al''_i)\sum_{k\ge 1,~  6\le k+l\le 8}\|\nabla^k g_i\|_{L^\infty(B_3)}^l, &\quad  \mbox{if }n\ge 9.  \end{cases}
\ee

Combing  \eqref{eq:tos-2} and  \eqref{eq:tos-6}, we see that
\begin{align}
\label{eq:tos-7'}
\mathcal{ E}_i(r) \ge \mathcal{\hat E}_i(r) - |\mathcal{ E}_i(r)-\mathcal{\hat E}_i(r)| \ge   \frac{1}{C} |W_{g_i}(0)|^2 \log m_i -O(r)
\end{align}
if $n=8$; and
\begin{align}
\label{eq:tos-7}
\mathcal{ E}_i(r) \ge   \frac{1}{C} |W_{g_i}(0)|^2 m_i^{\frac{2(n-8)}{n-4}} -O(\al_i'')\mathcal{G}_i -O(r)
\end{align}
if $n\ge 9$, where $O(r)=0$ if $n\ge 10$.
When $n=8,9$, by \eqref{eq:tos-7'} and \eqref{eq:tos-7} we have  $m_i^2 \mathcal{P}(r,u_i)=m_i^2 \mathcal{N}(r, u_i)+ \mathcal{ E}_i(r) \ge m_i^2 \mathcal{N}(r, u_i)-O(r)$.   By  \eqref{eq:tos-1} we have $ \mathcal{P}(r,\Gamma) \ge -Cr$. Thus \eqref{eq:sign restrict} follows and the conclusion (i) is proved.

If $n\ge 10$ and $|W_{g_i}(0)|^2 $ satisfies \eqref{eq:r2} for  large $C^*>0$, by  \eqref{eq:tos-7} we see that
\[
\mathcal{ E}_i(r) \ge   \frac{1}{2C} |W_{g_i}(0)|^2 m_i^{\frac{2(n-8)}{n-4}}\ge 0
.\] Hence, $m_i^2 \mathcal{P}(r,u_i)=m_i^2 \mathcal{N}(r, u_i)+ \mathcal{ E}_i(r) \ge m_i^2 \mathcal{N}(r, u_i)$ and  the conclusion (ii) follows from \eqref{eq:tos-1}.

When $r\le \frac{1}{2}$, by Remark \ref{rem:convergence} we have  $\mathcal{P}(r,\Gamma)\le C$ for some $C>0$ depending only on $n,A_1,A_2,A_3$. It follows that $ \mathcal{ E}_i(r) \le C- m_i^2 \mathcal{N}(r, u_i)$. Then (iii)  follows immediately from  \eqref{eq:tos-7'}, \eqref{eq:tos-7} and \eqref{eq-eq2}.

\end{proof}

\begin{prop} \label{prop:isolated to isolated simple} Given $p_i, G_i$, and $ h_i$ satisfying \eqref{p}, \eqref{G} and \eqref{H}  respectively, $\kappa_i$ satisfying \eqref{K} with $K_i$ replaced by $\kappa_i$,  let $0\le u_i\in C^4(B_3)$ solve both \eqref{eq:s1'} and \eqref{eq:Q-sub},  and assume \eqref{eq:IE-cond} holds. Suppose that $0$ is an isolated blow up point of $\{u_i\}$  with \eqref{eq:A_3} holds. Then $0$ is an isolated simple blow up point, if one of the three cases occurs: (i) $g_i$ is flat, (ii) $n\le 9$, (iii) $n\ge 10$ and \eqref{eq:r2} holds.
\end{prop}

\begin{proof} By Proposition \ref{prop:blow up a bubble}, $r^{4/(p_i-1)}\overline u_i(r)$ has precisely
one critical point in the interval $0<r<r_i$,
where $R_i\to \infty$ $r_i=R_iu_i(0)^{-\frac{p_i-1}{4}}$ as in Proposition \ref{prop:blow up a bubble}.
Suppose by contradiction that $0$ is not an isolated simple blow up point and let $\mu_i$ be the second critical point of $r^{4/(p_i-1)}\overline u_i(r)$. Then we must have
\[
\mu_i\geq r_i,\quad \dlim_{i\to \infty}\mu_i=0.
\]

Set $
v_i(x)=\mu_i^{4/(p_i-1)}u_i(\mu_i x)$ for $ x\in B_{3/\mu_i}$.
By the assumptions of Proposition \ref{prop:blow up a bubble}, $v_i$ satisfies
\begin{align*}
v_i(x)=\int_{B_{3/\mu_i}}\tilde G_{i}(x,y)\tilde \kappa_i(y)^{\tau_i} v_i(y)^{p_i}\,\ud y+\tilde h_i(x)
\\[2mm]
\lim_{i\to \infty}v_i(0)=\infty, \quad  |x|^{4/(p_i-1)}v_i(x)\leq A_3\quad \mbox{for } |x|<2/\mu_i \to \infty,
\end{align*}
\[
r^{4/(p_i-1)}\overline v_i(r)\mbox{ has precisely one critical point in } 0<r<1,
\]
and$
\frac{\mathrm{d}}{\mathrm{d}r}\left\{ r^{4/(p_i-1)}\overline v_i(r)\right\}\Big|_{r=1}=0,$
where $\tilde G_i=G_{i,\mu_i}$, $\tilde \kappa_i(y)=\kappa_i(\mu_i y)$,
$\tilde h_i(x)=\mu_i^{4/(p_i-1)}h_i(\mu_i x)
$ and $\overline v_i(r)=|\pa B_r|^{-1}\int_{\pa B_r}v_i$.
Therefore, $0$ is an isolated simple blow up point of $\{v_i\}$.

\textbf{Claim.} We have
\be \label{eq:converg2}
v_i(0) v_i (x) \to \frac{ac_{n}}{|x|^{n-4}} + ac_{n} \quad \mbox{in }C^3_{loc}(\R^n \setminus \{0\}),
\ee
where $a>0
$ is given in \eqref{eq:number a}.

First of all, by Proposition \ref{prop:upbound2} we have $\tilde h_i(e)\le v_i(e)\le C v_i(0)^{-1}$ for any $e\in \mathbb{S}^{n-1}$, where $C>0$ is independent of $i$. It follows from the assumption \eqref{H} on $h_i$ that
\[
v_i(0)\tilde h_i(x)\le C \quad \mbox{for all } |x|\le 2/\mu_i
\] and
\[
\|\nabla (v_i(0)\tilde h_i)\|_{L^\infty(B_{\frac{1}{9\mu_i}})} \le \mu_i \|v_i(0)\tilde h_i\|_{L^\infty(B_{\frac{1}{4\mu_i}})} \le C\mu_i.
\] Hence, for some constant $c_0\ge 0$, we have, along a subsequence,
\[
\lim_{i\to \infty}\|v_i(0)\tilde h_i(x)-c_0\|_{L^\infty(B_t)} =0, \quad \forall~t>0.
\]

Secondly, by Remark \ref{rem:convergence} and Proposition \ref{prop:upbound2} we have, up to a subsequence,
\be \label{eq:iso to isos1}
v_i(0)\int_{B_t}\tilde G_{i}(x,y)\tilde \kappa_i(y)^{\tau_i} v_i(y)^{p_i}\,\ud y \to \frac{a c_n}{|x|^{n-4}} \quad \mbox{in }C^3_{loc}(B_t \setminus \{0\}) \mbox{ for any } t>0,
\ee
where we used that $\tilde G_i(x,0)\to c_{n}|x|^{4-n}$.
Notice that for any $x\in B_{t/2}$
\[
Q''_i(x):=  \int_{B_{3/\mu_i}\setminus B_t}
\tilde G_{i}(x,y)\tilde \kappa_i(y) ^{\tau_i}v_i(y)^{p_i}\,\ud y \le C(n, A_1)\max_{\pa B_t} v_i.
\]
Since $\max_{\pa B_t} v_i\le Ct^{4-n} v_i(0)^{-1}$,  we have as in the proof of \eqref{eq:phigeos}, after passing to a subsequence,
\[
v_i(0) Q''_i(x)\to q(x) \quad \mbox{in }C^3_{loc}(B_t) \quad \mbox{as }i\to \infty
\] for some $q\in C^3(B_t)$. For any fixed large $R>t+1$,  it follows from \eqref{eq:iso to isos1} that \[
v_i(0)\int_{t \le |y|\le R}
\tilde G_{i}(x,y)\tilde \kappa_i(y)^{\tau_i}v_i(y)^{p_i}\,\ud y  \to 0
\]
as $i\to \infty$, since the constant $a$ is independent of $t$.  By the assumption \eqref{G} on $G_i$,  for any $x\in B_t$ and $|y|>R$, we have
\[
|\nabla_x \tilde G_{i}(x,y)| \le A_1 |x-y|^{3-n} \le  \frac{A_1}{R-t} |x-y|^{4-n} \le  \frac{A_1^2}{R-t}  \tilde G_{i}(x,y).
\]
Therefore, we have $
|\nabla q(x)| \le \frac{A_1^2}{R-t} q(x).$
By sending $R\to \infty$, we have $|\nabla q(x)| \equiv 0$ for any $x\in B_t$. Thus,
\[
q(x)\equiv q(0)\quad \mbox{for all } x\in B_t.
\]
 Since
\[
\frac{\mathrm{d}}{\mathrm{d}r}\left\{ r^{4/(p_i-1)}v_i(0)\overline v_i(r)\right\}\Big|_{r=1}=
v_i(0)\frac{\mathrm{d}}{\mathrm{d}r}\left\{ r^{4/(p_i-1)}\overline v_i(r)\right\} \Big|_{r=1}=0,
\]
we have, by choosing, for example, $t=2$ and sending $i$ to $\infty$, that
\[
q(0)+c_0=ac_{n}>0.
\]
Therefore, \eqref{eq:converg2} is proved.

It follows from \eqref{eq:converg2} and Lemma \ref{lem:test-poho} that
\[
\liminf_{i\to \infty} v_i(0)^2\mathcal{P}(r,v_i)=-(n-4)^2(n-2)a^2c_n^2 |\mathbb{S}^{n-1}|<0\quad \mbox{for all }0<r<1.
\]

On the other hand, by \eqref{eq:Q-sub} $v_i$ satisfies
\[
P_{\tilde g_i} v_i=c(n)\tilde \kappa_i^{\tau_i}v_i^{p_i} \quad \mbox{in }B_{3/\mu_i},
\]where $\tilde g_i(z)=g_i(\mu_i z)$. It is easy to see that \eqref{eq:IE-cond} is still correct with $G_i$ replaced by $\tilde G_i$. If $n\le 9$ or $g_i$ is flat, it follows from Proposition \ref{prop:one-side}  that
\be \label{eq:another-side}
\liminf_{r\to 0}\liminf_{i\to \infty} v_i(0)^2\mathcal{P}(r,v_i)\ge 0.
\ee
If $n \ge 10$, by \eqref{eq:IE-cond}, we have
\begin{align*}
U_{\lda}(x)&=\int_{B_{3/\mu_i}} G_{i,\mu_i}(x,y)\{U_\lda(y)^{\frac{n+4}{n-4}}+\mu_i^4 c_{\lda/\mu_i,i}'(\mu_i y) U_{\lda}(y)\}\,\ud y+\mu_i^{\frac{n-4}{2}}c_{\lda/\mu_i,i}''(\mu_ix) \\& = \int_{B_3} \tilde G_i(x,y) \{U_\lda(y)^{\frac{n+4}{n-4}}+ \tilde c_{\lda,i}'(y) U_{\lda}(y)\}\,\ud y+\tilde c_{\lda,i}''(x)   \quad \forall ~\lda\ge 1 ,~x\in B_3,
\end{align*}
where $c_{\lda,i}'(y):=\mu_i^4 c_{\lda/\mu_i,i}'(\mu_i y) $ and
\[
\tilde c_{\lda,i}''(x) = \int_{B_{3/\mu_i}\setminus B_3} G_{i,\mu_i}(x,y)\{U_\lda(y)^{\frac{n+4}{n-4}}+\mu_i^4 c_{\lda/\mu_i,i}'(\mu_i y) U_{\lda}(y)\}\,\ud y+\mu_i^{\frac{n-4}{2}}c_{\lda/\mu_i,i}''(\mu_ix).
\]
By the assumptions for $c'_{\lda,i}$ and $c''_{\lda,i}$, we have
\[
\tilde \Theta_i:=\sum_{i=0}^5 \|\lda^{-k} \nabla^k \tilde c'_{\lda, i}\|_{L^\infty(B_3)} \le \mu_i^4\Theta_i,
\]
and $\|\tilde c_{\lda,i}''\|_{C^5(B_2)}\le CA_2 \lda^{\frac{4-n}{2}}$, where $C>0$ depends only on $n,A_1,A_2$. Clearly, we have $|W_{\tilde g_i}(0)|^2=\mu_i^4 |W_{ g_i}(0)|^2$. Hence \eqref{eq:r2} is satisfied. By (ii) of Proposition \ref{prop:one-side}, we also have \eqref{eq:another-side}. We obtain a contradiction.

Therefore, $0$ must be an isolated simple blow up point of $u_i$
and the proof is completed.
\end{proof}

\begin{lem}\label{lem:isolated to isolated simple} Let $0\le u_i\in C^4(B_3)$ solve both \eqref{eq:s1'} and \eqref{eq:Q-sub} with $n\ge 10$,  and assume \eqref{eq:IE-cond} holds. For $\mu_i\to 0$, let
\[
v_i(x)=\mu_i^{\frac{4}{p_i-1}} u_i(\mu_i x).
\]
Suppose that $0$ is an isolated blow up point of $\{v_i\}$ and \eqref{eq:r2} holds. Then $0$ is also an isolated simple blow up point.
\end{lem}

\begin{proof} From the end of proof of Proposition \ref{prop:isolated to isolated simple}, we see that the condition \eqref{eq:r2} is preserved under the scaling $v_i(x)=\mu_i^{\frac{4}{p_i-1}} u_i(\mu_i x)$. Hence, the lemma follows from Proposition \ref{prop:isolated to isolated simple}.

\end{proof}

\section{Global analysis and proofs of main theorems}
\label{section:thm1.1}

Let $(M,g)$ be a smooth compact  Riemannian manifold of dimension $n\ge 5$. Suppose \eqref{condition:main2} holds.  Consider the equation
\be \label{eq:sub DE}
P_gu= c(n)u^p, \quad u \ge 0 \quad \mbox{on } M,
\ee
where $1<p\le \frac{n+4}{n-4}$.

As in the second order case \cite{Li-Zhu99}, we have

\begin{prop} \label{prop:reduction} Assume the above. For any given $R>0$ and $0\le \va<\frac{1}{n-4}$, there exist positive constants $C_0=C_0(M,g,R,\va)$, $C_1=C_1(M, g,R,\va)$ such that, for any smooth positive solution of \eqref{eq:sub DE} with
\[
\max_{M} u(X)\ge C_0,
\]
then $\frac{n+4}{n-4}-p<\va$ and  there exists a set of finite distinct points
\[
\mathscr{S}(u):=\{Z_1,\dots,Z_N\}\subset M
\] such that the following statements are true.

(i) Each $Z_i$ is a local maximum point of $u$ and
\[
\overline{\B_{\bar r_i}(Z_i)} \cap \overline{\B_{\bar r_j}(Z_j)}=\emptyset \quad \mbox{for }i\neq j,
\]
where $\bar r_i=R u(Z_i)^{(1-p)/4}$, and $\B_{r_i}(Z_i)$ denotes the geodesic ball in $B_2$ centered at $Z_i$ with radius $\bar r_i$

(ii) For each $Z_i$,
\[
\left\| \frac{1}{u(Z_i)}u\left(\exp_{Z_i}\left(\frac{y}{u(Z_i)^{(p-1)/4}}\right)\right)-\left(\frac{1}{1+ |y|^2}\right)^{\frac{n-4}{2}}\right\|_{C^4(B_{2R})} <\va.
\]

(iii) $u(X)\le C_1 dist_g(X,\{Z_1,\dots,Z_N\})^{-4/(p-1)}$ for all $X\in M$.

\end{prop}

\begin{rem}\label{rem:isolated on manifolds} On $(M,g)$, we say a point $\bar X\in M$ is an isolated blow up point for $\{u_i\}$ if there exists a sequence $X_i\in M$, where each $X_i$ is a local maximum point for $u_i$ and $X_i\to \bar X$, such that  $u_i(X_i)\to \infty$ as $i\to \infty$ and $u_i(X) \le C dist_g(X,X_i)^{-\frac{4}{p_i-1}} $ in $\B_{\delta}(X_i)$ for some constants $C,\delta>0$ independent of $i$. Under the assumptions that  $u_i$ is a positive solution of $P_{g}u_i=c(n)u_i^{p_i}$ with $0\le \frac{n+4}{n-4}-p_i\to 0$, $\mathrm{Ker} P_g=\{0\}$ and that the Green's function $G_g$ of $P_g$ is positive, it is easy to see that if $X_i\to \bar X\in M$ is an isolated blow up point of $\{u_i\}$, then in conformal normal coordinates centered at $X_i$, $0$ is an isolated blow up point of  $\{\tilde u_i(\exp_{X_i} x)\}$, where the exponential map is with respect to conformal metric $g_i=\kappa_i^{\frac{-4}{n-4}}g$, $\kappa_i>0$ is under control on $M$, and $\tilde u_i=\kappa_i u_i$; see Remark \ref{rem:blow}.
Since in Theorem  \ref{thm:final-a} and the sequel those conditions  will always be assumed,  the notation of isolated simple blow up points on manifolds is understood  in conformal normal coordinates.
\end{rem}

\begin{prop}\label{prop:ruling out accumulation}   If
one of the three cases happens: (i) $(M,g)$ is locally conformally flat, $n\le 9$, or $n\ge 10$ and $|W_g|^2>0$ everywhere,
then, for $\va>0$, $R>1$ and any solution \eqref{eq:sub DE} with $\max_M u>C_0$, we have
\[
|Z_1-Z_2|\ge \delta^*>0 \quad \mbox{for any }Z_1,Z_2\in \mathscr{S}(u), ~Z_1\neq Z_2,
\]
where $\delta^*$ depends only on $M,g$.

\end{prop}

\begin{proof} Since \eqref{condition:main2} is assumed, the proof is similar to that of  Proposition 5.2 of Li-Zhu \cite{Li-Zhu99}. Indeed, by the Green representation in conformal normal coordinates we can apply results in previous sections.
\end{proof}

Theorem \ref{thm:energy} is a part of the following theorem.

\begin{thm}\label{thm:final-a} Let $u_i\in C^4(M)$ be a sequences of positive solutions of $P_g u_i=c(n)u_i^{p_i}$ on $M$, where $0\le  (n+4)/(n-4)-p_i \to 0$ as $i\to \infty$. Assume the hypotheses in Proposition \ref{prop:ruling out accumulation}, then
\[
\|u_i\|_{H^2(M)} \le C,
\]
where $C>0$ depending only on $M,g$. Furthermore, after passing to  a subsequence,  $\{u_i\}$ is uniformly bounded or has only isolated simple blow up points and the distance between any two blow up points is bounded below by some positive constant depending only on $M,g$.

\end{thm}

\begin{proof}The last statement follows immediately from Proposition \ref{prop:reduction}, Proposition \ref{prop:ruling out accumulation} and Proposition \ref{prop:isolated to isolated simple}. Consequently,  it follows from \eqref{eq:cor-energy-2} and Lemma \ref{lem:error} and Proposition \ref{prop:upbound2} and Proposition \ref{prop:har} that $\int_{M} u_i^{\frac{2n}{n-4}}\,\ud vol_g\le C$. By the Green's representation and standard estimates for Riesz potential, we have  the $H^2$ estimates.
\end{proof}

\begin{proof}[Proof of Theorem \ref{thm:main theorem}]
By Proposition \ref{prop:reduction}, it suffices to consider that $p$ is close to $\frac{n+4}{n-4}$.  suppose by contradiction that there exists  a sequences of positive solutions $u_i\in C^4(M)$ of $P_g u_i=c(n)u_i^{p_i}$ on $M$, where $p_i\to (n+4)/(n-4)$ as $i\to \infty$, such that $\max_{M}u_i\to \infty$. By Theorem  \ref{thm:final-a}, let $X_i\to \bar X\in M$ be an isolated simple blow up point of $\{u_i\}$; see Remark \ref{rem:isolated on manifolds}.

If $n\ge 8$ and $|W_g|^2>0$ on $M$, we immediately obtain a contradiction to item (iii) of Proposition \ref{prop:one-side}.
Next, we consider $n\le 9$ or $(M,g)$ is locally conformally flat.

It follows from Proposition \ref{prop:one-side} that, in the $g_{\bar X}$-normal coordinates centered at $\bar X$,
\[
\liminf_{r\to 0} \mathcal{P}(r, c(n)G)\ge 0,
\]
where $g_{\bar X}$ a conformal metric of $g$ with $\det g_{\bar X}=1$ in an open ball $B_{\delta}$ of the $g_{\bar X}$-normal coordinates,  $G(x)=G_{g_{\bar X}}(\bar X,\exp_{\bar X} x)$ and $G_{g_{\bar X}}$ is the Green's function of $P_{g_{\bar X}}$.
On the other hand, if $n=5,6,7$ or $(M,g)$ is locally conformally flat, by Theorem \ref{thm:positive mass} and Lemma \ref{lem:test-poho} we have
\[
\mathcal{P}(r, c(n)G) <-A \quad \mbox{for small }r,
\]
where $A>0$ depends only on $M,g$. We obtain a contradiction.

If $n=8,9$, we have $W_g(\bar X)=0$. In view of Remark \ref{rem:positive mass},  we have
\[
\lim_{r\to 0}\mathcal{P}(r,G)=\begin{cases} -2 \dashint_{\mathbb{S}^{n-1}} \psi(\theta), &\quad n=8, \\
-\frac{5}{2}A, &\quad n=9,
\end{cases}
\]
where $\psi(\theta)$ and $A$ are as in Remark \ref{rem:positive mass}.
If the positive mass type theorem holds for Paneitz operator in dimension $n=8,9$, we obtain $\lim_{r\to 0}\mathcal{P}(r,G)<0$. Again, we derived  a contradiction.

Therefore, $u_i$ must  be uniformly bounded and the proof is completed.
\end{proof}

\appendix

\section{Local estimates for solutions of linear integral equations}

Let $\om_2\subset \subset  \om_1$ be a bounded open set in $\R^n$, $n\ge 5$. For $x,y\in \om_1 $, let $G(x,y)$ satisfy \eqref{G} with $B_3$ replaced by $\om_1$
and let  $0\le h\in C^4( \om_1)$ satisfy
\be \label{apdx:H}
\sup_{ \om_2} h \le A_2 \inf_{\om_2} h
\quad \mbox{and} \quad
\sum_{j=1}^4r^j|\nabla^j h(x)| \le A_2 \|h\|_{L^\infty(B_r(x))}
\ee
for all $x\in \om_2$ and $0<r<dist(\om_2,\pa\om_1)$.  We recall some local estimates for solutions  of the integral equation
\be \label{apdx:main}
u(x)=\int_{\om_1} G(x,y)V(y)u(y)\,\ud y+h(x) \quad \mbox{for } x\in \om_1.
\ee

\begin{prop}\label{prop:har} Assume as above.  Let $0\le V\in L^{\infty}(\om_1)$, and let $0\leq u\in C^0(\om_1)$ be a solution of \eqref{apdx:main}.
Then  we have
\[
\sup_{\om_2} u\le C\inf_{\om_2} u,
\]
where $C>0$ depends only on $n, A_1, A_2$, $\om_1$, $\om_2$ and $\|V\|_{L^\infty(\om_1)}$.
\end{prop}

\begin{proof} It follows from some simple modification of the proof of Proposition 2.3  of \cite{JLX3}. In fact, the third line of \eqref{G} is not needed.

\end{proof}

\begin{prop}\label{prop:local estimates} Suppose the hypotheses in Proposition \ref{prop:har}. Then $u\in C^3(B_2)$ and
\[
\|u\|_{C^3(\om_2 )} \le C\|u\|_{L^\infty(\om_1)},
\]
where $C>0$ depends only on $n, A_1,A_2$, the volume of $\om_2$, $dist(\om_2,\pa \om_1)$ and $\|V\|_{L^\infty(\om_1)}$.

If $V\in C^{1}(\om_1)$, then $u\in C^{4}(\om_2)$ and for any $\om_3\subset \subset \om_2$ we have
\[
\|\nabla ^4u\|_{L^\infty(\om_3)} \le C\|u\|_{L^\infty(\om_1)},
\]
where $C>0$ depends only on $n, A_1,A_2$, the volume of $\om_2$, $dist(\om_2,\pa \om_1)$, $dist(\om_3,\pa \om_2)$ and $\|V\|_{C^{1}(B_3)}$.

\end{prop}

\begin{proof} Let $f:=Vu$. If $k<4$,  we have
\[
\nabla^k u(x) =\int_{\om_1} \nabla^k_xG(x,y)f(y)\,\ud y+\nabla^kh(x) \quad \mbox{for }x\in \om_2,
\]
and thus
\begin{align*}
|\nabla^k u(x)| &\le A_1 \|f\|_{L^\infty(\om_1)} \int_{\om_1} |x-y|^{n-4+k}\,\ud y +|\nabla^k h(x)|  \\&
\le C( \|u\|_{L^\infty(\om_1)} +\|h\|_{L^\infty(\om_1)}).
\end{align*}
Since $u$ and $V$ are nonnegative, we have $0\le h(x)\le u(x)$. We proved the first conclusion.

Let $\om_3\subset \subset \om_2$.  Without loss of generality, we may assume $\pa \om_2\in C^1$. If $V\in C^1(\om_1)$, we have $f\in C^1(\bar \om_2)$. By the third line of \eqref{G}, we see
\[
\nabla_x G(x,y)=-\nabla_y G(x,y) +(O^{(3)}(|x|)-O^{(3)}(|y|) )|x-y|^{4-n}+O^{(3)}(|x-y|^{5-n}).
\] We have for $x\in \om_3$ and $1\le j\le n$,
\begin{align*}
\nabla_{x_j} \nabla^3 u(x)=& \int_{\om_2} \nabla_{x_j} \nabla_x^3 G(x,y) (f(y)-f(x))\,\ud y -f(x)\int_{\pa \om_{2}}\nabla_y^3 G(x,y) \nu_j\,\ud S(x) \\&
+f(x)O( \int_{\om_2} |x-y|^{1-n}\,\ud y)+  \int_{\om_1\setminus \om_{2}}  \nabla_{x_j} \nabla_x^3 G(x,y) f(y)\,\ud y+ \nabla_{x_j} \nabla^3  h(x),
\end{align*}
where $\nu=(\nu_1,\dots, \nu_n)$ denotes the outward normal to $\pa \om_2$.
By \eqref{G} and \eqref{apdx:H}, the proof follows immediately.
Hence, we complete the proof.

\end{proof}

\bigskip

\noindent Y.Y. Li

\noindent 
 Department of Mathematics, Rutgers University,\\
110 Frelinghuysen Road, Piscataway, NJ 08854, USA\\[1mm]
Email: \textsf{yyli@math.rutgers.edu}

\medskip

\noindent J. Xiong

\noindent School of Mathematical Sciences, Beijing Normal University\\
Beijing 100875, China\\[1mm]
Email: \textsf{jx@bnu.edu.cn}

\end{document}